\documentclass[12pt]{article}

\usepackage[dvipsnames]{xcolor}
\usepackage{hyperref}
\hypersetup{
colorlinks,
  linkcolor  = Blue,
  citecolor  = Blue,
  urlcolor   = Orange,
  filecolor  = magenta,
}
\usepackage{amsfonts,amsthm}
\usepackage{amssymb,latexsym,amsmath}
\usepackage{bbm} 
\usepackage{enumerate}
\usepackage{layout}
\usepackage[all,cmtip]{xy}
\usepackage{graphics}
\usepackage[utf8]{inputenc}
\usepackage{graphicx}
\usepackage{calc}
\usepackage{color}
\usepackage{mathrsfs}
\usepackage{mathtools}
\usepackage{slashbox}
\usepackage[nottoc]{tocbibind}

\usepackage{authblk} 

\usepackage{multirow}

\newcommand\fin{\hfill $\Diamond$}
\newcommand\fini{\hfill $\clubsuit$}

\DeclareMathOperator{\End}{End}
\DeclareMathOperator{\Aut}{Aut}

\DeclareMathOperator{\Ima}{Im}
\DeclareMathOperator{\Ker}{Ker}
\DeclareMathOperator{\Coker}{Coker}
\DeclareMathOperator{\sym}{\mathbf{Sym}}
\DeclareMathOperator{\x}{\mathbf{x}}
\DeclareMathOperator{\e}{\epsilon}

\newtheorem{theorem}{Theorem}[section]
\newtheorem{proposition}[theorem]{Proposition}
\newtheorem{lemma}[theorem]{Lemma}
\newtheorem{corollary}[theorem]{Corollary}

\theoremstyle{definition}

\newtheorem{example}[theorem]{Example}
\newtheorem{definition}[theorem]{Definition}

\title{\textcolor{Blue}{On first-species counterpoint theory}}
\author[1]{Juan Sebastián Arias-Valero\thanks{\url{jsariasv1@gmail.com}}}
\author[2]{Octavio Alberto Agustín-Aquino\thanks{\url{octavioalberto@mixteco.utm.mx}}}
\author[1]{Emilio Lluis-Puebla\thanks{\url{lluisp@unam.mx}}}
\affil[1]{Departamento de Matemáticas, Universidad Nacional Autónoma de México, Ciudad de México, México}
\affil[2]{Instituto de Física y Matemáticas, Universidad Tecnológica de la Mixteca, Huajuapan de León, México}

\date{\vspace{-7ex}}

\begin{document}

\maketitle

\begin{quote}
\begin{small}
\textbf{Abstract:} We generalize first-species counterpoint theory to arbitrary \textit{rings} and obtain some new \textit{counting and maximization results} that enrich the theory of admitted successors, pointing to a \textit{structural approach}, beyond computations. The generalization encompasses an \textit{alternative theory} of contrapuntal intervals. We also propose several \textit{variations} of the model that intend to deepen into its principles. The original \textit{motivations} of the theory, as well as all \textit{technical passages}, are carefully reviewed so as to provide a complete exposition. 
\end{small}
\end{quote}

\begin{quote}
\begin{small}
\end{small}
\end{quote}

\begin{quote}
\begin{small}
\end{small}
\end{quote}

\tableofcontents

\section{Introduction}

This article is a theoretic exercise on a model of \textit{first-species counterpoint} and the mathematical support of the musicological and computational analysis in \cite{Musicology}. This kind of counterpoint is the simplest one and the didactic basis of \textit{Renaissance counterpoint}, as taught by Johann J. Fux in \cite{Fux}. The model was introduced by Mazzola in \cite{Inicount}, where the discussion is confined to the case when the ground ring is $\mathbb{Z}_{12}$; a ring that can be used to model the algebraic behavior of the twelve intervals between tones in the \textit{chromatic scale} of Western musical tradition. Then, the model was re-exposed with some additional computational results by Hichert in \cite[Part~VII]{MazzolaTopos}. Further generalizations to the case when the rings are of the form $\mathbb{Z}_{n}$ were considered in \cite{OAAAthesis,Junod}, and then included in a collaborative compendium of mathematical counterpoint theory and its computational aspects \cite{Octavio}. The motivation for such a generalization to $\mathbb{Z}_n$ was the existence of \textit{microtonal scales} with more than twelve tones, which have been used for making real music \cite{octaviotod}. 

The contributions of this article to the previous theory are the following.

\begin{itemize}
\item Complete re-exposition of the model with a careful review of all its principles, clarifying important points, like the contrapuntal symmetry definition (Section~\ref{model}). 
\item Generalization of the model to \textit{noncommutative rings}. This offers the possibility of defining counterpoint notions on noncommutative rings of generalized intervals, which can be worth exploring in musical terms; see Example~\ref{noncomex} and Section \ref{handnc}. 
\item Some theoretic advances on previous computational approaches. They include the determination of self-complementary dichotomies with the \textit{groupoid of intervals} (Section~\ref{sec3.1}), a criterion for the strength of dichotomies (Lemma~\ref{userigid}), the \textit{counting formulas} for successors sets cardinalities (Section~\ref{seccountfor}), a \textit{maximization criterion} for these cardinalities (Section~\ref{secmax}), and the derivative admitted successors computation, \textit{by hand}, for Renaissance counterpoint (Section~\ref{handnc}).
\item  Some \textit{variations} of the model. First, our generalization includes an alternative model of contrapuntal intervals, namely the \textit{product ring} $R\times R$, which takes the place of the dual numbers ring $R[\epsilon]$. These structures only differ in their products; the first one regards intervallic variations as elements of $R$, whereas the second one treats these variations as infinitesimals; see Section~\ref{countint}. The other two variations intend to deepen into the second principle of the model, and coincide in the case of Renaissance counterpoint (Section~\ref{vars}).
\end{itemize}

We organize this paper as follows. We start with the basic definitions: symmetries (natural operations in music) in Section~\ref{sec1}, hierarchy of dichotomies and polarities (consonance/dissonance partitions and their symmetries) in Section~\ref{sec2}, and  contrapuntal intervals rings in Section~\ref{sec3}. In particular, in Section \ref{sec3.4}, we prove that several dichotomies are strong. Then, in Sections~\ref{alternation}, \ref{locchar}, and \ref{variety} we discuss the principles that lead to a formal definition of first-species counterpoint in Section~\ref{sec5.3}. The remaining Section~\ref{model} develops the mathematical details of the principles. In Section~\ref{secsimp}, we simplify this definition, by reducing it to the case when the cantus firmus is $0$ and use a smaller group of symmetries. Based on this simplification, we prove that, for the dual numbers case, it is always possible to find successors of a given consonant interval. This result is the \textit{Little theorem of counterpoint} (Section~\ref{little}). A crucial lemma (Lemma~\ref{gencount}) leads us to the counting formulas and the maximization criterion in Section~\ref{seccountfor}. With these tools, in Section~\ref{hand}, we obtain the admitted successors of consonant intervals in Renaissance counterpoint. In Section~\ref{handnc}, we explore a noncommutative notion of counterpoint. In Section~\ref{loccharequiv}, we prove that our contrapuntal symmetry condition 2 is equivalent to Mazzola's original one in the commutative case, although it is weaker in general. Finally, in Section~\ref{vars} we discuss two conceptual variations of the model and close this article by proving that they coincide in the case of the Renaissance dichotomy. The appendix includes the proof of a variation of the \textit{rearrangement inequality}, which is crucial for the maximization criterion.

Throughout this paper, $R$ denotes an arbitrary ring with unity, not necessarily commutative. When needed, commutativity hypothesis on $R$ is explicitly stated. Similarly, modules considered in this paper are right $R$-modules, unless we indicate otherwise.

\section{Ring symmetries}\label{sec1}

\textit{Affine homomorphisms are the natural correspondences that occur in music}. They are the formalization that encompasses musical transformations like transposition and inversion in the ring $\mathbb{Z}_{12}$ of tones.\footnote{The ring $\mathbb{Z}_{12}$ models the twelve tones of the chromatic scale and, at the same time, the twelve intervals between these tones. The distinction is usually clear from the context.} 

Let $R$ be a ring. Given two $R$-modules $M$ and $N$, an \textit{affine homomorphism} from $M$ to $N$ is the composite $e^a\circ f$, where $f:M\longrightarrow N$ is an $R$-homomorphism, $a\in N$, and $e^a:N\longrightarrow N:x\mapsto x+a$ is the \textit{translation} associated with $a$. We write a typical affine homomorphism $e^a\circ f$ as $e^af$, for short. Observe that the composition rule for affine homomorphisms is $(e^af)\circ (e^bg)=e^{f(b)+a}fg$.

The monoid (with respect to the composition) $\End_R(M)$ of $R$-endomorphisms of an $R$-module $M$ is a ring with the usual sum of homomorphisms. However, the monoid of affine endomorphisms of $M$ is not a ring since the distributivity fails, except when $M$ is the trivial module. We denote by $\Aut_R(M)$ the group of $R$-automorphisms of an $R$-module $M$. 

An affine homomorphism $e^af:M\longrightarrow N$ is an isomorphism if and only if $f$ is an $R$-isomorphism.\footnote{First, suppose that $e^b g:N\longrightarrow M$ is the inverse of $e^a f$. Then $(e^af)(e^bg)=e^{f(b)+a}fg=id_N$ and $(e^bg)(e^a f)=e^{g(a)+b}gf=id_M$. Therefore, $fg=id_N$, $gf=id_M$, and $b=-g(a)$. Conversely, if $fg=id_N$ and $gf=id_M$,  then $(e^af)(e^{-g(a)}g)=id_N$ and $(e^{-g(a)}g)(e^af)=id_M$.} Thus, affine automorphisms of $R$ are of the form $e^af$, where $f\in \Aut_R(R)$ and $a\in R$. On the other hand, recall that there is a ring isomorphism
\begin{equation*}
\End_R(R)\cong R,
\end{equation*}
which identifies an element $r\in R$ with the $R$-endomorphism obtained by left multiplication with $r$. This isomorphism restricts to a group isomorphism  
\begin{equation*}
\Aut_R(R)\cong R^*,
\end{equation*}
which establishes a bijective correspondence between $R$-automorphisms of $R$ and invertible elements of $R$. From this discussion we conclude that an affine automorphism of $R$ is the function associated with the linear polynomial $bx+a$, where $b\in R^*$ and $a\in R$. 

In what follows, we will call affine automorphisms of $R$  \textbf{symmetries} of $R$ and denote the group of symmetries of a ring $R$ by $\sym(R)$. We use the notation $e^ab$ for the symmetry associated with $bx+a$. Two symmetries $e^ab$ and $e^{a'}b'$ are equal if and only if $a=a'$ and $b=b'$. Observe that $e^01$ is the identity symmetry and that the composition of symmetries is given by the formula
\begin{equation}\label{Eq:CompositionOfSymmetries}
e^ab\circ e^{a'}b'=e^{ba'+a}bb'.
\end{equation}
The inverse of a symmetry $e^ab$ is $e^{-b^{-1}a}b^{-1}$.
\section{Basic motivations and definitions}\label{sec2}

In Renaissance counterpoint we divide the ring $\mathbb{Z}_{12}$ of \textbf{intervals} into two disjoint subsets, namely the set $K$ of \textit{consonances} and the set $D$ of \textit{dissonances}. The consonances are unison, minor third, major third, perfect fifth, minor sixth, and major sixth. The dissonances are minor second, major second, perfect fourth, tritone, minor seventh, and major seventh. The corresponding mathematical definitions are 
\[K=\{0,3,4,7,8,9\}\text{ and }D=\{1,2,5,6,10,11\}.\] 
A composition of \textit{first-species counterpoint} consists of two voices, \textit{cantus firmus} and \textit{discantus}, whose notes have the same length and where each interval between the voices is a consonance. See \cite[pp.~19-29]{Fux} and \cite[Section~2]{Musicology} for details.

The germ of Mazzola's model is the observation that there is a \textit{unique} symmetry $p$ of the ring $\mathbb{Z}_{12}$, namely $e^25$, that sends consonances to dissonances. In fact, 
\[e^25(0)=2,\ e^25(3)=5,\ e^25(4)=10,\ e^25(7)=1, \ e^25(8)=6, \ e^25(9)=11.\]
We postpone the proof of the uniqueness to Section~\ref{sec3.4}. Moreover, the conditions that $e^2 5$ is the unique symmetry in $\sym(\mathbb{Z}_{12})$ sending $K$ to $D$ and that $7K$ is a multiplicative monoid \textit{characterize} the partition $\{K,D\}$ according to \cite[Section~13.1]{Beau}. 

Up to now, we have the intervals, consonances, and dissonances of counterpoint, but we need to model the voices in a composition. We achieve this by means of the \textbf{countrapuntal intervals} ring $\mathbb{Z}_{12}[\x]$, which consists of all linear polynomials $c+d\x$, with $c,d\in \mathbb{Z}_{12}$, in an indeterminate $\x$ satisfying certain relation.\footnote{This relation is $\x^2=\alpha \x$ for some constant $\alpha\in\{0,1\}$. See Section~\ref{sec3} for details.} A contrapuntal interval $c+d\x$, represents a pitch class $c$, of the cantus firmus, together with the interval $d$ between $c$ and the pitch class $c+d$ from the superior discantus. As already said, in a piece of Renaissance counterpoint, it is mandatory that $d$ be a consonance. For instance, the  contrapuntal interval $2+7\x$ in $\mathbb{Z}_{12}[\x]$ comes from the following musical example. 

\includegraphics{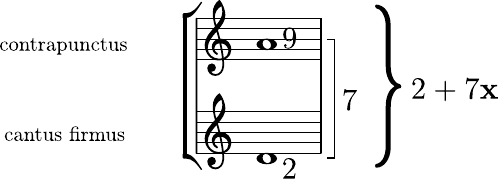}

We have a partition of the ring $\mathbb{Z}_{12}[\x]$ into \textbf{contrapuntal consonances and dissonances}, namely $\{K[\x],D[\x]\}$, where $X[\x]$ consists of all $c+d\x$ with $d\in X$ for $X=K,D$. We would want properties for this partition analogous to those of $\{K,D\}$. Certainly, the symmetry $e^{2\x} 5$ of $\mathbb{Z}_{12}[\x]$ is a quite natural\footnote{Precise statements of this naturality claim are Propositions~\ref{unique0} and \ref{unique0ori}.} extension of $e^2 5$ sending $K[\x]$ to $D[\x]$ since $e^{2\x}5$ is simple and acts on the interval part $d$ of a contrapuntal interval $c+d\x$ just as $e^25$. But in this case \textit{it is not the unique} sending $K[\x]$ to $D[\x]$. For example, $e^{2\x+1} 5$ also does. However, there is an induced \textit{local uniqueness property} that characterizes the partition $\{K[\x],D[\x]\}$, as discussed in Section~\ref{locchar}.

The first-species compositions deal with finite sequences of contrapuntal consonances with cantus firmus and discantus in a given \textit{scale}. However, not all sequences are valid since counterpoint has some \textit{rules}. In the model to be exposed in the following sections, we aim to predict when a given contrapuntal consonances $\eta$ can follow another one $\xi$, by means of the \textit{admitted successor} concept. We base the latter on the \textit{alternation principle} (Section~\ref{alternation}), the local uniqueness property (Section~\ref{locchar}), and the \textit{variety} principle of counterpoint (Section~\ref{variety}). 

\subsection{Dichotomies}\label{sec3.1}

We start to develop our theory by formalizing the common properties of the rings of intervals $\mathbb{Z}_{12}$ and of $\mathbb{Z}_{12}[\x]$ that are relevant for counterpoint.

Let $R$ be a ring.
\begin{itemize}
\item A partition $\{K,D\}$ of $R$ is a \textbf{dichotomy} of $R$ if $|K|=|D|$. Note that a finite ring $R$ has dichotomies if and only if its cardinality is even. 
\item A \textbf{self-complementary dichotomy}\footnote{This definition is essentially the same of an \textbf{autocomplementary} dichotomy in \cite[Definition~92]{MazzolaTopos}.} of $R$ is a triple $(K,D,p)$, where $\{K,D\}$ is a partition of $R$, $p$ is a symmetry of $R$, and $p(K)=D$. Note that necessarily $\{K,D\}$ is a dichotomy of $R$.
\end{itemize}

The triples $(K,D,e^25)$ and $(K[\x],D[\x],e^{2\x}5)$ from Renaissance counterpoint are self-complementary dichotomies. The dichotomy $\{\{0,1,3,6,8,11\},\{2,4,5,7,9,10\}\}$ of $\mathbb{Z}_{12}$ is not part of a self-complementary dichotomy of $\mathbb{Z}_{12}$ as proved at the end of this section. 

We would like to construct self-complementary dichotomies other than the basic one of Renaissance counterpoint, so as to create new notions of counterpoint. Given an arbitrary finite ring $R$ of even cardinality, we could start with a symmetry $p$ of $R$ and aim to construct a self-complementary dichotomy $(K,D,p)$. We pick some $k_1\in R$ to be in $K$, and define $d_1:=p(k_1)$, with $d_1$ to be in $D$. Then, if $|R|>2$, we take some $k_2\in R\setminus \{k_1,d_1\}$ to be in $K$ and define $d_2:=p(k_2)$ to be in $D\setminus \{d_1\}$, and so on. Note that this process produces a dichotomy $\{K,D\}$ provided \textit{$p$ has no fixed points} and is finite since $R$ is. We say that a symmetry $p$ without fixed points is a \textbf{derangement}.

The following proposition contains some concluding properties of self-complementary dichotomies.

\begin{proposition}\label{propsymdich1} Each self-complementary dichotomy $(K,D,p)$ of $R$ has the following properties:
\begin{enumerate}
\item The identity $p(D)=K$ holds, so $(D,K,p)$ is a self-complementary dichotomy.
\item The symmetry $p$ is a derangement. 
\end{enumerate}
\end{proposition}
\begin{proof}     
1. Since $p$ is a bijection, the properties of the inverse image imply that $p(D)=p(R\setminus K)=R\setminus p(K)=R\setminus D=K$.
2. Let $x\in R$. If $x\in K$, then $x\neq p(x)\in D$. If $x\in D$, then, by 1, $x\neq p(x)\in K$.
\end{proof}

Thus, \textit{self-complementary dichotomies are strongly related to derangements}.

The following is a short reflection about the problem of deciding when a dichotomy $\{K,D\}$ of a ring $R$ is not self-complementary. 

Let $X\subseteq R$. We associate with $X$ a categorical \textbf{groupoid of intervals} $\mathcal{G}(X)$ as follows. Its set of objects is $X$. For each $x,y\in X$, there is a unique morphism from $x$ to $y$, namely the triple $(x,y,y-x)$. We define the composition by  \[(y,z,z-y)\circ (x,y,y-x)=(x,z,z-y+y-x)=(x,z,z-x).\] The identities are of the form $(x,x,0)$. The inverse of $(x,y,y-x)$ is $(y,x,x-y)$. 

Given a self-complementary dichotomy $(K,D,e^ab)$ of $R$, we have an induced groupoid isomorphism $F:\mathcal{G}(K)\longrightarrow \mathcal{G}(D)$ sending\footnote{It is the unique possible definition if the correspondence on objects is $e^ab$.} $(x,y,y-x)$ to $(e^ab(x),e^ab(y),b(y-x))$, which acts linearly on intervallic variations. Hence, \textit{a self-complementary dichotomy induces a linear correspondence between the intervals of $K$ and $D$}. 

This fact can be used to prove that the dichotomy $\{K,D\}$ with $K=\{0,1,3,6,8,11\}$ and $D=\{2,4,5,7,9,10\}$ is not self-complementary. In fact, the sequence $(11,0,1)$ of elements of $K$ has the maximum length among sequences $(x_1,\dots ,x_n)$ such that $x_{i+1}-x_i=1$ for $i=0,\dots n-1$. If a symmetry $e^ab$ sends $K$ to $D$, then the groupoid isomorphism $F$ sends $(11,0,1)$ to a sequence $(d_1,d_2,d_3)$ in $D$ with maximum length among sequences $(y_1,\dots ,y_n)$ such that $y_{i+1}-y_{i}=b$. But the sequences of maximum length with $y_{i+1}-y_{i}=b$ are $(4,5)$ and $(9,10)$(for $b=1$), $(5,4)$ and $(10,9)$(for $b=11$), $(4,9,2,7)$ (for $b=5$), and $(7,2,9,4)$ (for $b=7$), whose length is not $3$; a contradiction. 

\subsection{Strong dichotomies and polarities}\label{sec3.2}

We formalize the additional uniqueness property of the Renaissance dichotomy $(K,D,e^25)$ with the following definitions. Let $R$ be a ring. A self-complementary dichotomy $(K,D,p)$ of $R$ is a \textbf{strong dichotomy} of $R$ if $p$ is the unique symmetry of $R$ such that $p(K)=D$. In such a case, we also say that $p$ is a \textbf{polarity}.

As we will show in the next sections, once we have a strong dichotomy of a ring (made up of generalized intervals), we have an induced self-complementary dichotomy of the contrapuntal intervals ring, and an associated theory of admitted successors. Thus, strong dichotomies lead to new \textit{counterpoint worlds} for composing non-traditional counterpoint.

We need to determine whether a given self-complementary dichotomy $(K,D,p)$ is strong. We approach this problem by determining the number of symmetries sending $K$ to $D$. Given another symmetry $q$ with $q(K)=D$, note that 
\[q^{-1}\circ p(K)=q^{-1}(D)=K\]
and hence $q^{-1}\circ p$ is in the stabilizer $\theta(K)$, where
\[\theta(K):=\{g \in \sym(R)\ |\ g(K)=K\}.\] 
This suggests the right action of $\theta(K)$, by composition, on the set $\sym(K,D)$ of symmetries of $R$ sending $K$ to $D$, as shown in the following diagram.
\[\begin{array}{cccc}
\circ :&\sym(K,D) \times \theta(K)& \longrightarrow& \sym(K,D)\\
 &(q,g) &\longmapsto & q\circ g
\end{array}\]
As we have observed, every $q\in \sym(K,D)$ is in the same orbit of $p$. Also, the stabilizer of $p$ under this action is the identity since $p\circ g=p$ implies $g=id_R$. Hence, the following proposition.

\begin{proposition}\label{corstab}
Let $(K,D,p)$ be a self-complementary dichotomy of a ring $R$. There is a bijective correspondence between $\sym(K,D)$ and $\theta(K)$. The bijection sends an element $g\in \theta(K)$ to $ p\circ g$.
\end{proposition}
\begin{proof}
The function $\theta(K)\longrightarrow \sym(K,D): g\mapsto p\circ g$ is surjective since the action is transitive and is injective since the stabilizer of $p$ is trivial.
\end{proof}

Thus, we have reduced the study of $\sym(K,D)$ to that of the stabilizer group $\theta(K)$ of $K$. Besides, we have determined when a self-complementary dichotomy is strong, as established in the following corollary. We say that a partition $\{K,D\}$ of $R$ is \textbf{rigid} if $\theta(K)$ is the trivial group. Similarly, we say that a subset $K$ of $R$ is \textbf{rigid} if $\theta(K)$ is trivial. Note that in a rigid partition $\{K,D\}$ of $R$ both $K$ and $D$ are rigid.
\begin{corollary}\label{rig}
A self-complementary dichotomy $(K,D,p)$ is a strong dichotomy if and only if $\{K,D\}$ is rigid.
\end{corollary}

Now, our objective is to prove a pair of lemmas that help us to decide whether a given dichotomy is rigid, by discarding non-identity symmetries that could be in $\theta(K)$. These lemmas could be a first approximation to a more conceptual argument.\footnote{This conceptual argument can be related to the groupoid of intervals of $K$, briefly introduced in Section~\ref{sec3.1}.} Consider the left action $\cdot$, by multiplication, of the group $R^*$ on $R$. 

\begin{lemma}\label{userigid}
Let $\{K,D\}$ be a dichotomy of a ring $R$. Suppose that $a\in R$ and that $C$ is either $K$ or $D$. If there is $r\in R$ such that 
\begin{enumerate}
\item the orbit $R^*r$ of $r$ under the action $\cdot$ is contained in $C$ and 
\item $r+a\notin C$,
\end{enumerate}
then $e^ab\notin \theta(K)$ for each $b\in R^*$. 
\end{lemma}
\begin{proof}
If $r+a\notin C$ for $r$ as above, then for each $b\in R^*$, $b(b^{-1}\cdot r)+a\notin C$ with $b^{-1}\cdot r\in R^*r\subseteq C$, and hence $e^ab\notin \theta(C)=\theta(K)$. 
\end{proof}

In particular, if $r=0$, we obtain the following result.

\begin{lemma}\label{disc0}
Suppose that $C$ is either $K$ or $D$ and that $0\in C$. Each symmetry of the form $e^ab$ with $a\in R\setminus C$ is not in $\theta(K)$. In particular, $e^ab\notin \theta(K)$.
\end{lemma}

Before using these lemmas to check that some dichotomies are strong, we first study a possible way to construct self-complementary dichotomies that are good candidates to strong ones.

\subsection{Constructing strong dichotomies: quasipolarities}\label{sec3.3}

A symmetry $p$ of a ring $R$ is \textbf{involutive} if $p\circ p=id$. If $p=e^ab$, by Equation \eqref{Eq:CompositionOfSymmetries}, this condition is equivalent to $e^{ba+a} b^2=e^0 1$, that is, to $ba+a=0$ and $b^2=1$. In the next proposition we prove that all polarities are involutive, so if we want a polarity, we could start with an involutive derangement, then define a suitable self-complementary dichotomy following the strategy in Section~\ref{sec3.1}, and finally check whether the uniqueness property holds. 
 
\begin{proposition}\label{involut} If a self-complementary dichotomy $(K,D,p)$ is strong, then $p$ is involutive.
\end{proposition}
\begin{proof}     
If $(K,D,p)$ is a self-complementary dichotomy and $p$ has the uniqueness property, then $p\circ p(K)=p(D)=K$ by 1 in Proposition~\ref{propsymdich1}, and hence $p\circ p=id$ by Corollary~\ref{rig}. 
\end{proof}

For example, the Renaissance polarity $e^2 5$ is involutive. The converse of this proposition does not hold since, in the Renaissance dichotomy $(K[\x],D[\x],e^{2\x}5)$, $e^{2\x}5$ is involutive, but the dichotomy is not strong as already observed.

We call involutive derangements \textbf{quasipolarities}. A \textbf{quasipolarization} of $R$ is a self-complementary dichotomy $(K,D,p)$ of $R$ where $p$ is involutive. Therefore, in view of Proposition~\ref{propsymdich1}, a self-complementary dichotomy $(K,D,p)$ is a quasipolarization if and only if $p$ is a quasipolarity.

Note that not all self-complementary dichotomies are quasipolarizations. For example, the self-complementary dichotomy $(\{0,2,4,6,8,10\},\{1,3,5,7,9,11\},e^1)$ of $\mathbb{Z}_{12}$ is not a quasipolarization since $e^1\circ e^1=e^2\neq id$. 

Thus, the construction of quasipolarizations is a good beginning if we want a strong dichotomy, though quasipolarizations need not be strong dichotomies.

\subsection{Examples of strong dichotomies}\label{sec3.4}

In this section, we use Lemmas \ref{userigid} and \ref{disc0} to open up a series of counterpoint worlds. First, consider the case when $R$ is the commutative ring $\mathbb{Z}_{12}$ and different examples of quasipolarizations that we will check to be strong dichotomies. The set of orbits of the action of the group $\{1,5,7,11\}$ of invertible elements of $\mathbb{Z}_{12}$ on $\mathbb{Z}_{12}$ is 
\[\{\{0\}, \{1,5,7,11\}, \{2,10\}, \{3,9\}, \{4,8\}, \{6\}\}.\]

In the following example we prove the claim that the Renaissance quasipolarization is a strong dichotomy.
\begin{example}[\textbf{Renaissance counterpoint}]
Consider the dichotomy $(K,D,e^25)$, where $K=\{0,3,4,7,8,9\}$ and $D=\{1,2,5,6,10,11\}$. Let us prove that the dichotomy is strong. According to Corollary~\ref{rig}, it is enough to show that our dichotomy is rigid, or equivalently, that $\theta(K)=\{e^01\}$. In fact, by Lemma~\ref{disc0}, $e^ab\notin \theta(K)$ for each symmetry $e^ab$ with $a\in D$. The cases when $a\in K$ remain. If $a=0$, then since $5\times 7=11\notin K$, $7\times 7=1\notin K$, and $11\times 7=5\notin K$, we conclude that the symmetries of the form $e^0b$ with $b=5,7,9$ are not in $\theta(K)$. Finally, since
\[3+8=8+3=11\notin{K},\ 4+9=9+4=1\notin{K}\text{, and } 3+7\notin{K},\]
Lemma~\ref{userigid} implies that all symmetries $e^ab$ with $a\in \{3,4,7,8,9\}$ are not in $\theta(K)$. We have exhausted all possibilities, except the identity, and hence $\theta(K)$ is trivial.\fin
\end{example}

\begin{example}[\textbf{Scriabin's mystic chord}, see {\cite{Scriabinworld}}]
Consider the dichotomy $(K,D,e^5 11)$, where $K=\{0,2,4,6,9,10\}$ and $D=\{1,3,5,7,8,11\}$. Let us show that $e^5 11$ is a polarity by using the same strategy of the preceding example. 

As before, $e^ab\notin \theta(K)$ for each symmetry $e^ab$ with $a\in D$, and we need to discard the cases when $a\in K$. If $a=0$, then since $5\times 4=8\notin K$, $7\times 9=3\notin K$, and $11\times 4=8\notin K$, the symmetries of the form $e^0b$ with $b=5,7,11$ are not in $\theta(K)$. Also, since
\[6+2=2+6=8\notin{K},\ 5+4=9\notin{D},\ 2+9=11\notin{K}\text{, and } 10+10=8\notin{K},\]
Lemma~\ref{userigid} implies that all symmetries $e^ab$ with $a\in \{2,4,6,9,10\}$ are not in $\theta(K)$. Hence, $\theta(K)$ is trivial. \fin
\end{example}

The next computation establishes the existence of at least a strong dichotomy in $\mathbb{Z}_{2k}$ for each $k\geq 3$. These dichotomies are the beginning of a lot of microtonal counterpoint worlds, which were first studied in \cite{OAAAthesis}. Recall that all invertible elements of $\mathbb{Z}_{2k}$ are odd since they are coprime with $2k$.

\begin{example}[Cf. {\cite[Proposition 2.1]{Octavio}} and {\cite[Proposition 2.6]{OAAAthesis}}] 
Let\footnote{Thanks to the hypothesis $k\geq 3$, the six distinct elements $0$, $1$, $3$, $-1$, $2$, and $4$ are in $\mathbb{Z}_{2k}$.} $k\geq 3$. In $\mathbb{Z}_{2k}$, consider the dichotomy $(K,D,e^{-1}(-1))$, where \[K=\{0,1,3,..., 2k-5,2k-3\}  \text{ and } D=\{-1,2k-2,2k-4,...,4,2\}.\] 

We already know (Lemma~\ref{disc0}) that $e^ab\notin  \theta(K)$ for each symmetry $e^ab$ with $a\in D$. If $a=0$, then $b(-1)=-b\neq -1$ with $-b$ odd whenever $b\in\mathbb{Z}_{2k}^*\setminus \{1\}$ and hence $-b\notin D$, so the symmetries of the form $e^0b$ with $b\in\mathbb{Z}_{2k}^*\setminus \{1\}$ are not in $\theta(D)=\theta(K)$. Also, since all orbits of even numbers, except the orbit $\{0\}$ of $0$, under the action of $\mathbb{Z}_{2k}^*$ on $\mathbb{Z}_{2k}$ are contained in $D$, the equations
\[2+1=3\notin{D}, ...,2+(2k-5)=2k-3\notin{D}\text{, and } 4+(2k-3)=1\notin{D},\]
imply (Lemma~\ref{userigid}) that all symmetries $e^ab$ with $a\in \{1,3,..., 2k-5,2k-3\}$ are not in $\theta(K)$. Hence, $\theta(K)$ is trivial. \fin
\end{example}

Now, a noncommutative example.

\begin{example}[\textbf{A noncommutative strong dichotomy}]\label{noncomex}
Let $R$ be the noncommutative ring of all upper triangular matrices $2\times 2$ with entries in $\mathbb{Z}_2$. The symmetry $e^I$, where $I$ is the identity matrix, is a quasipolarity, so we can construct self-complementary dichotomies with the procedure of Section \ref{sec3.1}. A possible choice is the dichotomy with $K=\{\mathbf{0},A_1,A_2,A_3\}$, where 
\[A_1=
\begin{bmatrix}
 1 & 0 \\
      0  & 0
\end{bmatrix},\ 
A_2=\begin{bmatrix}
 1 & 1 \\
      0  & 0
\end{bmatrix},\text{ and } 
A_3=
\begin{bmatrix}
 1 & 1 \\
      0  & 1
\end{bmatrix}.\]
In this case, $R^*=\{I,A_3\}$ and $A_1$ and $A_2$ are invariant under the left action of $R^*$. Let us show that $K$ is rigid. As usual, $e^AB\notin \theta(K)$ for each symmetry $e^AB$ with $A\in D$. The cases when $A\in K$ remain. If $A=0$, then since $A_3 A_3=I\notin K$, the symmetry $e^{\mathbf{0}}A_3$ is not in $\theta(K)$. Also, 
\[A_2+A_1=A_1+A_2\notin{K}\text{, and } A_1+A_3\notin{K},\]
so, by Lemma~\ref{userigid}, all symmetries $e^AB$ with $A\in \{A_1,A_2,A_3\}$ are not in $\theta(K)$. Thus, $\theta(K)$ is trivial, and $(K,R\setminus K,e^I)$ is a strong dichotomy. The reader can obtain other examples of strong dichotomies from quasipolarities of $R$, whenever the chosen $K$ is not invariant under the left action of $A_3$, by using a similar argument.
\fin
\end{example}

Finally, note that there is no strong dichotomy of $\mathbb{Z}_4$ to define a counterpoint world with a four-tone scale (which can be thought of as a diminished seventh chord). In fact, the possible dichotomies $\{\{0,1\},\{2,3\}\}$, $\{\{0,2\},\{1,3\}\}$, and $\{\{0,3\},\{1,2\}\}$ are invariant under $e^1(-1)$, $e^2$, and $e^{-1}(-1)$, respectively, so they are not rigid or strong (Corollary~\ref{rig}). 

\subsection{The contrapuntal intervals ring}\label{sec3}

Given an arbitrary ring $R$, we can construct the polynomial ring $R[x]$ and the \textit{two-sided} ideal $\left\langle p(x) \right\rangle$ consisting of all (left or right) multiples of $p(x)$, where $p(x)$ is the polynomial $x^2-\alpha x$ for some fixed $\alpha=0,1$. Thus, we have the quotient ring\footnote{See \cite[p. 3]{Lam} for details.} $R[x]/\left\langle p(x)\right\rangle$, which we call \textit{the contrapuntal intervals ring} associated with $R$. If we denote by $\x$ the class of $x$, then each element of this ring can uniquely be written\footnote{In this representation we use the fact that $R$ can be regarded as a subring of $R[x]/\left\langle p(x)\right\rangle$. Also, the representation follows from the fact that each polynomial $f(x)$ in $R[x]$ can uniquely be written (division algorithm) as $q(x)p(x)+r(x)$, where either $r(x)=0$ or $deg(r(x))<deg(p(x))=2$.} as $c+d\x$, and hence it makes sense to denote the ring by $R[\x]$. Thus, $R[\x]$ consists of all linear polynomials $c+d\x$ subject to the relation $\x^2=\alpha \x$. The ring $R[\x]$ is noncommutative if $R$ is. 

Explicitly, the operations in $R[\x]$ are defined by
\[(c+d\x)+(c'+d'\x)=c+c'+(d+d')\x
\] 
and
\[(c+d\x)(c'+d'\x)=cc'+(cd'+dc'+dd'\alpha)\x.
\] 

In the case when $\alpha=0$, $R[\x]$ is the \textit{dual numbers} ring $R[\epsilon]$. In the case when $\alpha=1$, $R[\x]$ is isomorphic to the \textit{product} ring $R\times R$ by means of the assignment 
\[R\times R\longrightarrow R[\x]:(r,s)\mapsto (r,s-r).\] 
Both variations share the same structure as Abelian groups, but their differ in their products. In the dual numbers case, intervallic variations are \textit{infinitesimals}, so $d\epsilon d' \epsilon=dd'\epsilon ^2=0$, whereas in the other case, intervallic variations are \textit{just elements of $R$}, so $d\x d' \x=dd'\x ^2=dd'\x$.

An element $c+d\x \in R[\x]$ is invertible if and only if $c$ and $c+d\alpha$ are invertible in $R$. In fact, if $(c'+d'\x)(c+d\x)=1$ and $(c+d\x)(c'+d'\x)=1$, then $c'c+(c'd+d'c+d'd\alpha)\x=1$ and $cc'+(cd'+dc'+dd'\alpha)\x=1$, so $c$ is invertible and $cd'+dc'+dd'\alpha=0$. From the last equation we deduce the following ones.
\begin{align*}
(c+d\alpha)d'+dc' & =0 \\
(c+d\alpha)d'\alpha+d\alpha c' & =0 \\
(c+d\alpha)d'\alpha+d\alpha c'+cc' & =cc'\\
(c+d\alpha)(c'+d'\alpha)&=1 
\end{align*}
Similarly, $c+d\alpha$ is left-invertible. Reciprocally, if $c$ and $c+d\alpha$ are invertible, define $c'=c^{-1}$ and $d'=-(c+d\alpha)^{-1}dc'$. Thus, $c'+d'\x$ is a two-sided inverse for $c+d\x$.

Alternatively, in the case when $\alpha=1$, by the isomorphism between $R[\x]$ and $R\times R$, the inverse of $c+d\x$ also has the description $c^{-1}+((c+d)^{-1}-c^{-1})\x$.

We have the following proposition regarding dichotomies of $R[\x]$ induced by dichotomies of a ring $R$. Given subsets $X,Y\subseteq R$, we define
\[X+ Y \x =\{c+d\x\in R[\x]\ |\ c\in X\text{ and } d\in Y\}.\]
Also, we denote by $X[\x]$ the set $R+X\x$.

\begin{proposition}\label{dualexts} Let $\{K,D\}$ be a dichotomy of a ring $R$.
\begin{enumerate}
\item The pair $\{K[\x],D[\x]\}$ is a dichotomy of $R[\x]$.
\item If $(K,D,e^ab)$ is a self-complementary dichotomy of $R$, then $(K[\x],D[\x],e^{a\x}b)$ is a self-complementary dichotomy of $R[\x]$.
\item If $(K,D,e^ab)$ is a quasipolarization of $R$, then $(K[\x],D[\x],e^{a\x}b)$ is a quasipolarization of $R[\x]$. 
\end{enumerate}
As previously commented, strong dichotomies need not induce strong dichotomies on contrapuntal intervals.  
\end{proposition}
\begin{proof}
1. Exercise. 
2. This follows from 1 and the equalities
\begin{align*}
e^{a\x}b(K[\x])&=\{e^{a\x}b(r+k\x)\ |\ r\in R\text{ and }k\in K\}\\
 &=\{br+(bk+a)\x\ |\ r\in R\text{ and }k\in K\}\\
 &=bR+(e^ab(K))\x\\
 &=R+D\x\\
 &=D[\x].
\end{align*}
As to the fourth equality above, note that $bR=R$ since $b$ is invertible.
3. This follows from 2 and the equation
\[e^{a\x}b\circ e^{a\x}b=e^{(ba+a)\x}b^2=e^0 1,\]
which is a consequence of the identity $e^a b\circ e^a b=e^0 1$, equivalent to $ba+a=0$ and $b^2=1$. 
\end{proof}

\section{A mathematical model of first-species counterpoint}\label{model}

We start with a \textit{strong dichotomy} $(K,D,p)$ of a ring $R$, where $p=e^ab$, and construct the contrapuntal intervals ring $R[\x]$, which models the two voices of first-species counterpoint. Then, we consider \textbf{progressions} $(\xi,\eta)$ of contrapuntal consonances in $K[\x]$ and aim to determine when they are valid for counterpoint. The three principles in Sections~\ref{alternation}, \ref{locchar}, and \ref{variety} serve this purpose.

\subsection{Alternation}\label{alternation}

This principle intends to formalize the idea of \textit{tension/resolution} in counterpoint. It requires the existence of a symmetry $g\in \sym(R[\x])$ such that $\xi\in g(D[\epsilon])$ and $\eta\in g(K[\epsilon])$. In such a case, we say that $(\xi,\eta)$ is \textbf{polarized}. Thus, $(\xi,\eta)$ is the deformation, by means of $g$, of a progression from a dissonance to a consonance;  see Figure~\ref{counterpoint}. Under a musical intuition, $\xi$ moves to $\eta$ since dissonances resolve to consonances. However, how do we ensure that deformed consonances and dissonances are consonances and dissonances on their own right? In the following section we solve this problem. In Section~\ref{pols} we give a simple characterization of polarized progressions. 
\begin{figure}
\centering
\def\svgscale{0.8}
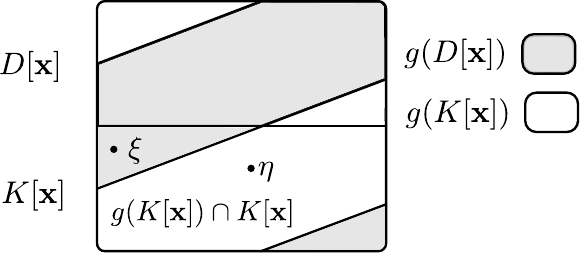
\caption{Here, $g$ is a deformation symmetry and $\eta$ an admitted successor of $\xi$.}
 \label{counterpoint}
\end{figure}

\subsection{Local characterization of consonances and dissonances}\label{locchar}

As observed in Section~\ref{sec2}, the strength property of the Renaissance dichotomy characterizes it. In general, given a strong dichotomy $(K,D,p)$ as above, the induced dichotomy of the contrapuntal intervals is not strong, but it has the following local strength property in each \textbf{fiber} $I_z$ of $R[\x]$, where $I_z=z+R\x$ and $z\in R$, as shown in Theorem~\ref{indxdich}: 
\begin{itemize}
\item There is a unique symmetry $p^z[\x]\in \sym(R[\x])$ of the form $e^{u+v\x}c$, defined by $p^z[\x]=e^{(1-b)z+a\x}b$, such that $p^z[\x](z+K\x)=z+D\x$.
\end{itemize}
 
In the case of the Renaissance strong dichotomy $(K,D,e^25)$ of $\mathbb{Z}_{12}$, this condition locally determines the induced dichotomy $\{K[\x],D[\x]\}$. This means that the following two conditions \textit{characterize} $\{z+K\x,z+D\x\}$ among partitions $\{z+K'\x,z+D'\x\}$ of $I_z$:  
\begin{enumerate}
\item[(i)] The symmetry $p^z[\x]$ is the unique $p'\in \sym(R[\x])$ of the form $e^{u+v\x}c$ such that $p'(z+K'\x)=z+D'\x$.
\item[(ii)] The set $7K'$ is a multiplicative monoid.
\end{enumerate} 
In fact, (i) is equivalent to the condition that $p$ is the unique symmetry of $\mathbb{Z}_{12}$ sending $K'$ to $D'$, so according to the characterization of the Renaissance dichotomy in Section~\ref{sec2}, the conditions (i) and (ii) imply $K'=K$ and $D'=D$. Consequently, if we want to ensure that a certain partition $\{z+K'\x,z+D'\x\}$ of $I_z$ offers a relaxed consonance/dissonance notion, we only require (i); otherwise we back to $\{z+K\x,z+D\x\}$. Thus, for a general strong dichotomy, we locally characterize contrapuntal dissonances and consonances with (i). 

Now, by alternation, $\xi$ (respectively $\eta$) is in a half of the partition 
\begin{equation}\label{defoloc}
\{g(K[\x])\cap I_z,g(D[\x])\cap I_z\},
\end{equation}
where $z$ is the cantus firmus of $\xi$ (respectively $\eta$). Therefore, we require the condition (i) on the partition in Equation~\eqref{defoloc}. In Theorem~\ref{equivcond}, we prove that this particular condition is equivalent to the equation
\begin{equation}\label{two}
p^z[\x](g(K[\x])\cap I_z)=g(D[\x])\cap I_z.
\end{equation}
Mazzola requires it only for the cantus firmus $z$ of $\xi$.\footnote{In Section~\ref{vars}, we discuss other possibilities, like also requiring (i) for the cantus firmus of $\eta$, which need not coincide with $z$. On the other hand, this condition is \textit{weaker} than Mazzola's second requirement \cite[Definition~95]{MazzolaTopos} $p^z[\x](g(K[\x]))=g(D[\x])$, and they are equivalent if $R$ is commutative, as shown in Section~\ref{loccharequiv}.}

\subsection{Variety}\label{variety}

In this model, the variety principle of counterpoint \cite[p.~21]{Fux} corresponds to the condition that there is a maximum of alternations from $\xi$, that is, the cardinality of $g(K[\x])\cap K[\x]$ is maximum among all $g\in \sym(R[\x])$ such that 1. $\xi \in g(D[\x])\cap K[\x]$ (alternation) and 2. Equation~\eqref{two} holds for the cantus firmus $z$ of $\xi$ (local dissonance).

\subsection{Defining first-species counterpoint}\label{sec5.3}

Now we can give the definition of admitted successor.   
\begin{definition}\label{def} Let $(K,D,p)$ be a \textit{strong dichotomy} of a ring $R$. 

\begin{itemize}
\item A \textbf{contrapuntal symmetry} for a consonance $\xi\in K[\x]$, with $\xi=z+k\x$, is a symmetry $g$ of $R[\x]$ such that 
\begin{enumerate}
\item $\xi\in g(D[\x])$,
\item $p^z[\x](g(K[\x])\cap I_z)=g(D[\x])\cap I_z$, and 
\item the cardinality of $g(K[\x])\cap K[\x]$ is maximum among all $g$ satisfying 1 and 2. 
\end{enumerate}
Note that the contrapuntal symmetry for a given consonance is not required to be unique. 
\item An \textbf{admitted successor} of a consonance $\xi\in K[\x]$ is an element $\eta$ of $g(K[\x])\cap K[\x]$ for some contrapuntal symmetry $g$. See Figure~\ref{counterpoint}.
\item Progressions $(\xi,\eta)$ usually occur in a subset $X$ of $R$, which we call \textbf{scale}. If $\eta$ is an admitted successor of $\xi$ we say that $(\xi,\eta)$ is \textbf{allowed}. If it does not happen and $(\xi,\eta)$ is polarized, we say that it is \textbf{forbidden}.
\end{itemize}
\fini
\end{definition}

Since the model chooses suitable pairs $(\xi,\eta)$ as allowed among all polarized pairs, we say that \textit{the model does not decide on non-polarized pairs}. Next we characterize polarized pairs and observe that, in the case of dual numbers ($\alpha=0$), they are all non-constant progressions. In the case $\alpha=1$, they could be more particular.

\subsection{Polarized progressions characterization}\label{pols}

Let $\xi=z+k\x$ and $\eta=z'+k'\x$, where $\xi,\eta\in K[\x]$ and $(K,D,p)$ is a strong dichotomy of a \textit{finite} ring $R$. We want to determine when there is $g\in R[\x]$ such that $z+k\x \in g(D[\x])$ and $z'+k\x\in g(K[\x])$. But the latter property is equivalent to $k\x \in e^{-z}g(D[\x])$ and $z'-z+k\x\in e^{-z}g(K[\x])$, so we reduce the problem to the case when $\xi=k\x$ and $\eta=y+k'\x$. Now, the following properties are equivalent. 
\begin{enumerate}
\item There is $g\in \sym(R[\x])$ such that $k\x \in g(D[\x])$ and $y+k'\x\in g(K[\x])$.
\item There is $g\in \sym(R[\x])$ such that $g(k\x) \in D[\x]$ and $g(y+k'\x)\in K[\x]$.
\item There is $e^{u+v\x}(c+d\x)\in \sym(R[\x])$ such that $(c+d\alpha)k+v \in D$ and $(c+d\alpha)k'+v+dy\in K$.
\item There are $e^{u+v\x}(c+d\x)\in \sym(R[\x])$ and $\kappa\in K$ such that $v=\kappa-(c+d\alpha)k'-dy$ and $\kappa+(c+d\alpha)(k-k')-dy \in D$.
\item There is $c+d\x$ invertible such that $e^{(c+d\alpha)(k-k')-dy}(K)\nsubseteq K$.
\item There is $c+d\x$ invertible such that $e^{(c+d\alpha)(k-k')-dy}(K)\neq K$.
\item There is $c+d\x$ invertible such that $(c+d\alpha)(k-k')-dy\neq 0$.
\end{enumerate}
The equivalence between 5 and 6 is given by the finiteness of $R$. 

Note that in turn 7 implies $k\x\neq y+k'\x$. In the dual numbers case ($\alpha=0$), we can prove that it is also a sufficient condition.

\begin{proposition} Let $\xi,\eta\in K[\e]$, where $(K,D,p)$ is a strong dichotomy of a finite ring $R$. The pair $(\xi,\eta)$ is polarized if and only if $\xi\neq\eta$. 
\end{proposition}
\begin{proof}
By 7 above, it only remains to prove that if $k\e\neq y+k'\e$, then there are $c\in R^*$ and $d\in R$ such that $c(k-k')-dy\neq 0$. We have two cases. \textit{First case:} $k-k'\neq 0$. Take $d=0$ and $c=1$. \textit{Second case:} $y\neq 0$ and $k-k'=0$. Take $d=1$ and any $c$.
\end{proof}

In the case when $\alpha=1$, this result is not true. Let us consider the case $R=\mathbb{Z}_{12}$. If $k-k'\neq 0$, we polarize $(\xi,\eta)$ as in the previous proof. However, if $y\neq 0$ and $k-k'=0$, we have two possibilities. First, if $y\neq 6$, then $2y\neq 0$ for $5+2\x$ invertible. Second, if $y=6$, then $dy=0$ for all $c+d\x$ invertible since $d$ is even always. Thus, the polarized pairs $(\xi,\eta)$ are all those such that $k\neq k'$, or $k=k'$ and $y\notin \{0,6\}$. In other words, the non-polarized progressions are all repetitions and parallelisms with $y=6$ (tritone skip).

\subsection{The structural role of polarities}\label{sec5.1}

Next, we prove that there is a unique symmetry $p^z[\x]\in \sym(R[\x])$ of the form $e^{u+v\x}c$ that interchanges contrapuntal consonances and dissonances in the fiber $I_z$. In particular, it leaves $I_z$ unchanged. In the case when $z=0$, $p^0[\x]$ is the extension $e^{a\x}b$ studied in Proposition~\ref{dualexts}. 

\begin{definition} Let $g$ be a symmetry of $R[\x]$ and $z\in R$. We say that $g$ is \textbf{$z$-invariant}
if it leaves invariant the fiber $I_z$, that is,  if $g(I_z)=I_z$. \fini
\end{definition}

We first compute some candidates to $p^z[\x]$. 

\begin{proposition}\label{charinv}
A symmetry $g$ of $R[\x]$, where $g=e^{u+v\x}(c+d\x)$, is $z$-invariant if and only if $u=(1-c)z$.
\end{proposition}
\begin{proof}
Note that 
\begin{align*}
g(z+R\x) &=  e^{u+v\x}(c+d\x)(z+R\x) \\
& =  cz+u+
((c+d\alpha)R+dz+v)\x  = cz+u+
R\x.
\end{align*}
Thus, $g$ is $z$-invariant if and only $cz+u=z$, the latter condition being equivalent to $u=(1-c)z$.
\end{proof}

We denote by $H_z$ the \textit{group}\footnote{In fact, this set is a group since it is an stabilizer one.} of all $z$-invariant symmetries of $R[\x]$. In the case when $z=0$, we simply write $H$ instead of $H_0$. The conjugation automorphism $e^z\circ (-)\circ e^{-z}$ of $\sym(R[\x])$ restricts to an isomorphism $H\longrightarrow H_z$ for each $z$, as established in the following proposition.

\begin{proposition}\label{conj}
For each $z\in R$, the conjugation homomorphism $e^z\circ (-)\circ e^{-z}:H \longrightarrow  H_z$ is an isomorphism of groups.
\end{proposition}
\begin{proof}
First, let us prove that the conjugation homomorphism has its images in $H_z$. If $h\in H$, then
\[
e^{z}\circ h\circ e^{-z} (z+R\x) = e^z\circ h(R\x)
= e^{z}(R\x)=z+R\x.
\]
Moreover, the inverse of $e^z\circ (-)\circ e^{-z}$ is $e^{-z}\circ (-)\circ e^z$.
\end{proof}

The \textit{uniqueness} in the following theorem offers an explanation of the structural role of the strong dichotomy $(K,D,p)$ with $p=e^ab$.

\begin{theorem}\label{indxdich}
Let $(K,D,e^ab)$ be a strong dichotomy of $R$. For each $z\in R$, there is a unique $z$-invariant symmetry $p^z[\x]$ of the form $e^{u+v\x}c$, defined by $p^z[\x]=e^z\circ e^{a\x}b\circ e^{-z}=e^{(1-b)z+ a\x}b$, such that 
\begin{equation*}
p^z[\x](z+K\x)=z+D\x.
\end{equation*}
\end{theorem}
\begin{proof}
Since symmetries $p'\in H_z$ of the form $e^{u+v\x}c$ such that $p'(z+K\x)=z+D\x$ are in correspondence (conjugation) with 
symmetries in $H$, of the same form,  sending $K\x$ to $D\x$, the proof of the theorem reduces to the case when $z=0$ (Proposition~\ref{unique0}). In that case, $p^0[x]$ is $e^{a\x}b$, so $p^z[\x]=e^z\circ e^{a\x}b\circ e^{-z}=e^{(1-b)z+ a\x}b$.
\end{proof}

\begin{proposition}\label{unique0}
Let $(K,D,e^ab)$ be a strong dichotomy of $R$. The symmetry $e^{a\x}b$ is the unique $0$-invariant symmetry of the form  $e^{u+v\x}c$ sending $K\x$ to $D\x$. 
\end{proposition}
\begin{proof}
First, by Proposition~\ref{charinv}, necessarily $u=0$. Now, $e^{v\x}c$ sends $K\x$ to $D\x$ if and only $e^vc(K)=D$. Hence, $e^vc=e^ab$ and $e^{u+v\x}c=e^{a\x}b$.
\end{proof}

\subsection{The local condition on the deformed dichotomy}\label{locchardefo}

\begin{lemma}\label{hachita}
Let $g$ be a symmetry of $R[\x]$. For each $z\in R$, there is $\gamma_z\in \sym(R)$ such that 
\begin{equation*}
g(X[\x])\cap I_z=z+\gamma_z(X)\x
\end{equation*}
for any subset $X$ of $R$. Concretely, if $g=e^{u+v\x}(c+d\x)$, we define \[\gamma_z=e^{v+dc^{-1}(z-u)}(c+d\alpha).\]
\end{lemma}
\begin{proof}
If $g=e^{u+v\x}(c+d\x)$, then
\[g(X[\x])\cap I_z=z+((c+d\alpha)X+v+dc^{-1}(z-u))\x=z+\gamma_z(X)\x.\]
\end{proof}

\begin{theorem}\label{equivcond}
Let $(K,D,p)$ be a strong dichotomy of $R$ and $g\in \sym(R[x])$. The following conditions are equivalent.
\begin{enumerate}
\item The symmetry $p^z[\x]$ is the unique $p'\in H_z$ of the form $e^{u+v\x}c$ such that 
\begin{equation}\label{locintdefo}
p'(g(K[\x])\cap I_z)=g(D[\x])\cap I_z.
\end{equation} 
\item The equation $p^z[\x](g(K[\x])\cap I_z)=g(D[\x])\cap I_z$ holds.
\item The equation $p\gamma_z=\gamma_zp$ holds, where $\gamma_z$ is as in Lemma~\ref{hachita}.
\end{enumerate}
\end{theorem}
\begin{proof}
Equation\eqref{locintdefo}, for $p'=e^{u+v\x}c\in H_z$, is equivalent to the equations $(e^{v}c)\gamma_z(K)=\gamma_z(D)$, $\gamma_z^{-1}(e^{v}c)\gamma_z(K)=D$, and $(e^{v}c)\gamma_z=\gamma_zp$, by Lemma \ref{hachita} and the fact that $p$ is a polarity. The last equation (and hence Equation~\eqref{locintdefo}) has a unique solution for $e^vc$ (respectively $p'$) and Equation~\eqref{locintdefo}, for $p'=p^z[\x]$, is equivalent to $p\gamma_z=\gamma_zp$, so 1 is equivalent to 3 and 2.
\end{proof}

\section{Simplified computation of admitted successors}\label{secsimp}

In this section we characterize the admitted successors of a contrapuntal consonance $z+k\x$ as the translations, by $z$, of the admitted successors of $k\x$ associated with contrapuntal symmetries in $H$. This offers and important simplification of the original computation.

First, we transfer the symmetries satisfying conditions 1, 2, and 3, and admitted successors, between the consonances $z+k\x$ and $k\x$. The basic tool is the \textbf{translation permutation} $e^{z}\circ (-):\sym(R[\x])\longrightarrow \sym(R[\x])$, whose inverse is $e^{-z}\circ (-)$.

\subsection{Transfer of the first condition}\label{i}
Let $z+k\x$ be a contrapuntal consonance. We define $S^{1}_{z+k\x}$ as the set of all $g\in \sym(R[\x])$ satisfying condition 1 in Definition~\ref{def}, namely $z+k\x \in g(D[\x])$. Note that the translation $e^{z}\circ (-)$ restricts to a bijection $S^1_{k\x}\longrightarrow S^1_{z+k\x}$. This means that $k\x\in g(D[\x])$ if and only if $z+k\x\in e^zg(D[\x])$. 

\subsection{Transfer of the second condition}\label{sec6.2}

We define $S^{2}_{z}$ as the set of all $g\in \sym(R[\x])$ satisfying condition 2 in Definition~\ref{def}, namely $p^z[\x](g(K[\x])\cap I_z)=g(K[\x])\cap I_z$. The translation $e^{z}\circ (-)$ restricts to a bijection $S^2_0\longrightarrow S^2_z$, as shown by the following equivalences. Recall that $p^0[\x]=e^{-z}\circ p^z[\x]\circ e^{z}$ by Theorem~\ref{indxdich}.   
\begin{align*}
p^0[\x](g(K[\x])\cap I_0)=g(D[\x])\cap I_0 & \Leftrightarrow\\
e^{-z}p^z[\x]e^z (g(K[\x])\cap I_0)=g(D[\x])\cap I_0 & \Leftrightarrow\\
 p^z[\x]e^z (g(K[\x])\cap I_0)=e^z(g(D[\x])\cap I_0) & \Leftrightarrow\\
 p^z[\x](e^z g(K[\x])\cap I_z)=e^zg(D[\x])\cap I_z
\end{align*}
In fact, $e^z$, as a bijection, commutes with intersections and $e^z(I_0)=I_z$. 
\subsection{Transfer of the third condition and admitted successors}\label{sec6.3}

Let $z+k\x$ be a contrapuntal consonance and $S_{z+k\x}$ the set of all $g\in \sym(R[\x])$ satisfying 1 and 2 in Definition~\ref{def}. Since the translation $e^{z}\circ (-)$ restricts to bijections between sets of symmetries satisfying 1 and 2, respectively (Sections \ref{i} and \ref{sec6.2}), then it restricts to a bijection $S_{k\x}\longrightarrow S_{z+k\x}$.

On the other hand, note that 
\begin{align*}
e^{z}(g(K[\x])\cap K[\x])     & =e^z g(K[\x])\cap e^z(K[\x])\\
                                           & =e^zg(K[\x])\cap K[\x].
\end{align*}
The second equation holds because $X[\x]$ (for any $X$) is invariant under transformations of the form $e^{y}$. In particular, this means that $|g(K[\x])\cap K[\x]|=|e^zg(K[\x])\cap K[\x]|$ and hence, in that sense, the bijection $S_{k\x}\longrightarrow S_{z+k\x}$ preserves the cardinality of the sets of the form $g(K[\x])\cap K[\x]$. 

Thus, $|g(K[\x])\cap K[\x]|$ is maximum, among all $g$ satisfying 1 and 2 in Definition~\ref{def}, for the consonance $k\x$, if and only $|e^zg(K[\x])\cap K[\x]|$ is for $z+k\x$. This means that $S_{k\x}\longrightarrow S_{z+k\x}$ restricts to a bijective correspondence between contrapuntal symmetries. 

In particular, $g(K[\x])\cap K[\x]$ is a set of admitted successors of $k\x$ if and only if $e^zg(K[\x])\cap K[\x]$ is for $z+k\x$. This implies the following lemma.

\begin{lemma}\label{transsuc}
The sets of admitted successors of $z+k\x$ can be computed as those of the form 
\[e^z(g(K[\x])\cap K[\x]),\] 
where $g$ is a contrapuntal symmetry for $k\x$. 
\end{lemma}

\subsection{Restricting symmetries}

Given a symmetry $g$, we next prove that $g(K[\x])=h(K[\x])$, where $h\in H$. Note that this immediately implies that $h$ is contrapuntal for $z+k\x$ whenever $g$ is. This suggests to draw our attention to the contrapuntal symmetries in $H$.

Recall that $H_z$ consists of all $0$-invariant symmetries of $R[\x]$, that is, the symmetries of the form $e^{v\x}(c+d\x)$ with $c,c+d\alpha\in R^*$ (Proposition~\ref{charinv}).

\begin{proposition}\label{hache}
Let $g$ be a symmetry of $R[\x]$. For each $z\in R$, there is $g_z\in H_z$ such that 
\begin{equation*}
g(X[\x])=g_z(X[\x])
\end{equation*}
for any subset $X$ of $R$. Concretely, if $g=e^{u+v\x}(c+d\x)$, then $g_z$ is $ge^w$, where $w=c^{-1}(z-u)-z$.
\end{proposition}
\begin{proof}
Note that the symmetry $g_z$ defined above is in $H_z$, by Proposition~\ref{charinv}, since $cw+u=(1-c)z$. Thus, 
\[g_z(X[\x])=ge^w(X[\x])=g(X[\x]).\]
\end{proof}

\subsection{Main theorem}
Our previous observations lead to a considerably simpler characterization of contrapuntal symmetries and admitted successors, only involving symmetries in $H$.

\begin{proposition}\label{minisimp} The admitted successors of a consonance $z+k\x\in K[\x]$ can be computed as the elements of the sets of the form
\begin{equation}
e^z(h(K[\x])\cap K[\x]),
\end{equation}
where $h\in H$ and 
\begin{itemize}
\item[(a)] $k\x\in h(D[\x])$, 
\item[(b)] $p^0[\x](h(K[x])\cap I_0)=h(D[x])\cap I_0$, and 
\item[(c)] the cardinality of $h(K[\x])\cap K[\x]$ is maximum among all $h\in H$ satisfying (a) and (b). 
\end{itemize}
\end{proposition}
\begin{proof}
By Lemma~\ref{transsuc}, it is enough to show that the collection of all sets of the form $g(K[\x])\cap K[\x])$, with $g$ contrapuntal symmetry for $k\x$, is equal to the set of all intersections $h(K[\x])\cap K[\x])$ with $h$ satisfying (a), (b), and (c) above.

If $g$ is a contrapuntal symmetry for $k\x$, then, since $g(K[\x])=g_0(K[\x])$ (Proposition~\ref{hache}), $g_0$ is a contrapuntal symmetry and, in particular, satisfies (a) and (b). Also, $g_0$ satisfies (c) because $H\subseteq \sym(R[\x])$. Moreover, $g(K[\x])\cap K[\x]=g_0(K[\x])\cap K[\x]$. 

Conversely, if $h$ satisfies (a), (b), and (c), we claim that $h$ is a contrapuntal symmetry and hence $h(K[\x])\cap K[\x]$ is an usual admitted successors set. To prove the claim, take $g$ satisfying 1 and 2 in Definition~\ref{def} for the consonance $k\x$. The symmetry $g$ satisfies (a), (b), and $g(K[\x])=g_0(K[\x])$, so 
\[|g(K[\x])\cap K[\x]|=|g_0(K[\x])\cap K[\x]|\leq |h(K[\x])\cap K[\x]|,\]
and $h$ is contrapuntal.   
\end{proof}

If we replace condition (b) in Proposition~\ref{minisimp} by its equivalent 3 in Lemma~\ref{equivcond}, we immediately obtain the main theorem.

\begin{theorem}\label{redu} The admitted successors of a consonance $z+k\x\in K[\x]$ can be computed as the elements of the sets of the form
\begin{equation}\label{transpose}
e^z(h(K[\x])\cap K[\x]),
\end{equation}
where $h=e^{v\x}(c+d\x)\in H$ and 
\begin{enumerate}
\item $k\x\in h(D[\x])$, 
\item $p\circ e^v(c+d\alpha)=e^v(c+d\alpha)\circ p$, and 
\item the cardinality of $h(K[\x])\cap K[\x]$ is maximum among all $h\in H$ satisfying 1 and 2. 
\end{enumerate}
\end{theorem}

This theorem says that, to compute the admitted successors of a consonance $z+k\x$, it is enough to do so for $k\x$, and then apply the transposition $e^z$ (Equation~\eqref{transpose}). In this simplification we only use symmetries in $H$ and reduce the condition 2 to a suitable commutativity.

\section{Counting formulas and maximization}\label{seccountfor}

Let $R$ be a \textit{finite} ring, $(K,D,p)$ a strong dichotomy of $R$ with $p=e^a b$, and $\{K[\x],D[\x]\}$ the induced dichotomy of $R[\x]$; see Proposition~\ref{dualexts}.

According to Theorem~\ref{redu}, it is desirable to count the number of elements of the sets of the form $h(K[\x])\cap K[\x]$, for $h\in H$, so as to choose suitable maximum cardinals. 

If $e^{v\x}(c+d\x)$, with $c,c+d\alpha \in R^*$, belongs to $H$ and $r+k\x \in K[\x]$, then 
\begin{equation*}
e^{v\x}(c+d\x)(r+k\x)=cr+((c+d\alpha)k+v+dr)\x.
\end{equation*}
Thus,
\begin{equation*}
e^{v\x}(c+d\x)(K[\x])=\bigsqcup\limits_{r\in R}cr+((c+d\alpha)K+v+dr)\x
\end{equation*}
and
\begin{equation}\label{presuc}
e^{v\x}(c+d\x)(K[\x])\cap K[\x]=\bigsqcup\limits_{r\in R}cr+(((c+d\alpha)K+v+dr)\cap K)\x
\end{equation}
because $e^0 c$ is a permutation of $R$, and hence 
\begin{equation}\label{counting}
|e^{v\x}(c+d\x)(K[\x])\cap K[\x]|=\sum\limits_{r\in R}|((c+d\alpha)K+v+dr)\cap K|.
\end{equation}
Moreover, we can assume that $\alpha=0$ in the last equation in view of the following ones.

\begin{align*}
\sum\limits_{r\in R}|((c+d\alpha)K+v+dr)\cap K| & = \sum\limits_{r\in R}\sum\limits_{k\in K}\chi_K((c+d\alpha)k+v+dr)\\
& = \sum\limits_{r\in R}\sum\limits_{k\in K}\chi_K(ck+v+d(\alpha k +r))\\
& = \sum\limits_{k\in K}\sum\limits_{r\in R}\chi_K(ck+v+d(\alpha k +r))\\
& = \sum\limits_{k\in K}\sum\limits_{r\in R}\chi_K(ck+v+dr)\\
& = \sum\limits_{r\in R}|(cK+v+dr)\cap K|
\end{align*}

\begin{lemma}\label{gencount} Let $K$ and $K'$ be subsets of a finite ring $R$ and $d\in R$. The equation
\begin{equation}\label{forgen}
\sum\limits_{r\in R}|(K'+dr)\cap K|=\rho\sum\limits_{\gamma\in \Coker(d\cdot-)}|K'_{\gamma}||K_{\gamma}|
\end{equation}
holds, where
$\Coker(d\cdot-)=R/ \Ima(d\cdot-)$,
\[X_{\gamma}=\{x\in X\ |\ [x]=\gamma\}\]
for each $X\subseteq R$, and $\rho=|\Ker(d\cdot-)|$.
\end{lemma}
\begin{proof}
\begin{align*}
\sum\limits_{r\in R}|(K'+dr)\cap K|
&= \rho\sum\limits_{r\in \Ima(d\cdot-)}|(K'+r)\cap K|\\
 &= \rho\sum\limits_{r\in \Ima(d\cdot-)} \sum\limits_{s\in K'} \chi_{K}(s+r)\\
 &= \rho \sum\limits_{s\in K'} \sum\limits_{r\in \Ima(d\cdot-)} \chi_{K}(s+r)\\
 &= \rho\sum\limits_{s\in K'} |K_{[s]}|\\
 &= \rho\sum\limits_{\gamma\in K'/\Ima(d\cdot-)} |K'_{\gamma}||K_{\gamma}|=\rho\sum\limits_{\gamma\in \Coker(d\cdot-)} |K'_{\gamma}||K_{\gamma}|.
\end{align*}
As to the last equality, note that if $\gamma$ is not in the image $K'/\Ima(d\cdot-)$ of $K'$ under the canonical projection onto the cokernel, then $|K'_{\gamma}|=0$.
\end{proof}

Let us point out a curious fact. In Lemma~\ref{gencount}, the cases $d=1$ and $d=0$ correspond to the formulas
\[\sum\limits_{r\in R}|(K'+r)\cap K|=|K'||K|\] 
and
\[\sum\limits_{r\in R}|K'\cap K|=|R||K'\cap K|\]
respectively; see \cite[Lemma~48]{MazzolaTopos}.

The following corollary illustrates the lemma in the important case when the ring is $\mathbb{Z}_{n}$.

\begin{corollary}\label{countint} Let $K$ and $K'$ be subsets of $\mathbb{Z}_{n}$ and $d\in \mathbb{Z}_{n}$. The equation
\begin{equation}\label{forint}
\sum\limits_{r=0}^{n-1}|(K'+dr)\cap K|=\rho\sum\limits_{i=0}^{\rho-1}|K'_{i}||K_{i}|
\end{equation}
holds, where 
\[X_{i}=\{x\in X\ |\ x \equiv i\pmod{\rho}\}\]
for each $X\subseteq \mathbb{Z}_{12}$ and $\rho=\gcd (d,n)$.
\end{corollary}
\begin{proof}
Note that for each ring $R$ (isomorphism and index theorems for groups)
\[|\Ker(d\cdot-)|=|R|/|\Ima(d\cdot-)|=|R/\Ima(d\cdot-)|=|\Coker(d\cdot-)|.\]
If $R=\mathbb{Z}_n$, then by \cite[Theorem 6.14]{Fraleigh}, $R/\Ima(d\cdot-)=\mathbb{Z}_n/d\mathbb{Z}_n=\mathbb{Z}_n/\rho\mathbb{Z}_n\cong \mathbb{Z}_{\rho}$, where $\rho=\gcd(d,n)$, so $|\Ker(d\cdot-)|=|\Coker(d\cdot-)|=\rho$. 
Moreover, $[x]=\gamma$ for $\gamma\in \Coker(d\cdot-)$ if and only if $x \equiv i\pmod{\rho}$ by identifying $\gamma$ and $i$ through the isomorphism $\Coker(d\cdot-)\cong \mathbb{Z}_{\rho}$.
\end{proof}

\begin{example}\label{vectors}
Let $(K,D,e^2 5)$ be the Renaissance strong dichotomy of $\mathbb{Z}_{12}$, with $K=\{0,3,4,7,8,9\}$ and $D=\{1,2,5,6,10,11\}$. The following table contains the particular values of the right-hand side of Equation~\eqref{counting}, by using Equation~\eqref{forint} with $K'=cK+v$, for $d$ ranging over $\mathbb{Z}_{12}$. After the table, we briefly justify each row.
\begin{center}
\begin{tabular}{|l|l|l|}
\hline

$d$ & $\rho$ &  $\sum\limits_{r=0}^{11}|(K'+dr)\cap K|$ \\ \hline
 0 & 12 & $12(|K'_0|+|K'_3|+|K'_4|+|K'_7|+|K'_8|+|K'_9|)$ \\ \hline
 1, 5, 7, 11 & 1 & 36 \\ \hline
 2, 10 & 2 & 36 \\ \hline
 3, 9 & 3 &  $3(|K'_0|3+|K'_1|2+|K'_2|)$ \\ \hline
 4, 8 & 4 & $4(|K'_0|3+|K'_1|+|K'_3|2)$ \\ \hline
 6 & 6 & $6(|K'_0|+|K'_1|+|K'_2|+|K'_3|2+|K'_4|)$ \\ \hline
\end{tabular}
\end{center}

\textit{Case} $d=0$. The number of elements in $K$ congruent to $i$ modulo $12$ is just given by the characteristic function of $K$ as a subset of $\mathbb{Z}_{12}$.

\textit{Case} $d=1$, $5$, $7$, $11$. Since all integers are congruent modulo 1, we obtain the expression $|K'||K|$ for the sum, whose exact value is $6\times 6$.  

The \textit{case} $d=2$, $10$ is quite interesting. The integers congruent to $0$ (respectively $1$) modulo $2$ are the even (respectively odd) ones. Now, there are three even numbers ($0$, $4$, and $8$) and three odd numbers ($3$, $7$, and $9$) in $K$. Thus, the counting formula becomes $2(|K'_0|3+|K'_1|3)$. But multiplying $K$ by $c$ (which is always odd because it is invertible) does not alter the parity of its elements, and adding $v$ to $cK$ does not alter the number of odd or even elements. For this reason, $|K'_0|=|K'_1|=3$ and the exact value of the counting formula is $2( (3\times 3)+ (3\times 3))$.

\textit{Case} $d=3$, $9$. In this case, $|K_0|=|\{0,3,9\}|=3$, $|K_1|=|\{4,7\}|=2$, $|K_2|=|\{8\}|=1$, and the remaining terms of the form $K_i$ are empty.

\textit{Case} $d=4$, $8$. Here $|K_0|=|\{0,4,8\}|=3$, $|K_1|=|\{9\}|=1$, $|K_3|=|\{3,7\}|=2$, and the remaining terms of the form $K_i$ are empty.

\textit{Case} $d=6$. Here $|K_0|=|\{0\}|=1$, $|K_1|=|\{7\}|=1$, $|K_2|=|\{8\}|=1$, $|K_3|=|\{3,9\}|=2$, $|K_4|=|\{4\}|=1$, and the remaining terms of the form $K_i$ are empty. \fin
\end{example}

Another important consequence of the equation 
\[\sum\limits_{r\in R}|(K'+dr)\cap K|
= \rho_{d}\sum\limits_{r\in \Ima(d\cdot-)}|(K'+r)\cap K|\]
in the proof of Lemma~\ref{gencount}
is that if $\Ima(d\cdot-)=\Ima(d'\cdot-)$, then 
\[\rho_{d}=|\Ker(d\cdot-)|=|R|/|\Ima(d\cdot-)|=|R|/|\Ima(d'\cdot-)|=\rho_{d'}\]
and hence 
\[\sum\limits_{r\in R}|(K'+dr)\cap K|
=\sum\limits_{r\in R}|(K'+d'r)\cap K|.\]

Moreover, in the case when $R=\mathbb{Z}_n$, we observe from Equation~\eqref{forint} that we can reduce the computation of the sums of the form $\sum\limits_{r\in R}|(K'+dr)\cap K|$ for all $d$ with the same $\rho$ (recall that $\rho=\gcd(d,n)$) to that of the sum $\sum\limits_{r\in R}|(K'+\rho r)\cap K|$, since these sums coincide.

\subsection{Main counting formulas}\label{commufor}

Now, according to Equation~\eqref{counting} (with $\alpha=0$), we will focus on the case when $K'=cK+v$. Let us assume that $d$ is in the \textit{center} $Z(R)$ of $R$. For each $d$, the sequence $(|K'_{\gamma}|)_{\gamma}$ is a \textit{rearrangement}\footnote{A rearrangement of a function (sequence) $f:S\longrightarrow X$ is a function of the form $f\circ \sigma$, where $\sigma$ is a permutation of $S$.} of $(|K_{\gamma}|)_{\gamma}$, where $\gamma$ ranges over $\Coker(d\cdot-)$. In fact, note first that
\begin{align*}
|K'_{[r]}|&=|\{e^vc(k)\ |\ k\in K\text{ and }[e^vc(k)]=[r]  \text{ in } \Coker(d\cdot-)\}|\\
&=|\{e^vc(k)\ |\ k\in K\text{, }ck+v-r \in dR\}|\\
&=|\{e^vc(k)\ |\ k\in K\text{, }k-(c^{-1}r- c^{-1}v)\in dR\}|\\
&=|\{e^vc(k)\ |\ k\in K\text{, }k-(e^vc)^{-1}(r)\in dR\}|\\
&=|\{e^vc(k)\ |\ k\in K\text{, }[k]=[(e^vc)^{-1}(r)]\text{ in } \Coker(d\cdot-)\}|\\
&=|\{k\in K\ |\ [k]=[(e^vc)^{-1}(r)]\}|=|K_{[(e^vc)^{-1}(r)]}|.
\end{align*}
Second, since $d\in Z(R)$, the function (actually a ring symmetry)
\[\begin{array}{cccc}
[e^vc]:&\Coker(d\cdot-) & \longrightarrow & \Coker(d\cdot-)\\
&{[r]} & \longmapsto     & [e^vc(r)]
\end{array}\]
is well defined (check) with inverse given by the function induced by $(e^vc)^{-1}$, so $[e^vc]$ and $[(e^vc)^{-1}]$ are permutations of the cokernel. Thus, regarding $(|K'_{\gamma}|)_{\gamma}$ and $(|K_{\gamma}|)_{\gamma}$ as functions from the cokernel to $\mathbb{N}$, the former is the composite of the latter with  $[(e^vc)^{-1}]$. Also, note that $[(e^vc)^{-1}]=[e^{c^{-1}v}c^{-1}]=e^{[c]^{-1}[v]}[c]^{-1}$.  This proves the main counting formulas.

\begin{theorem}[\textbf{Main counting formulas}]\label{mainform}
Let $R$ be a finite ring, $d\in Z(R)$, $K\subseteq R$, and $e^{v\x}(c+d\x)$ a symmetry in $H$. The equation
\begin{equation}\label{maincount}
|e^{v\x}(c+d\x)(K[\x])\cap K[\x]|=\rho\sum\limits_{\gamma\in \Coker(d\cdot-)}|K_{[(e^vc)^{-1}](\gamma)}||K_{\gamma}|=\rho\sum\limits_{\delta\in \Coker(d\cdot-)}|K_{\delta}||K_{[e^vc](\delta)}|
\end{equation}
holds, where $K_{\gamma}=\{x\in K\ |\ [x]=\gamma\}$ and $\rho=|\Ker(d\cdot-)|$. In particular, if $R=\mathbb{Z}_n$, the right-hand term of Equation~\eqref{maincount} coincides with 
\begin{equation}\label{maincountforint}
\rho\sum\limits_{i=0}^{\rho-1}|K_{i}||K_{e^vc(i)}|,
\end{equation}
where $K_{i}=\{x\in K\ |\ x \equiv i\pmod{\rho}\}$, $\rho=\gcd (d,n)$, and $e^vc$ is reduced modulo $\rho$.
\end{theorem}
\begin{proof}
Combine the previous discussion with Equation~\eqref{counting}, Lemma~\ref{gencount}, and Corollary~\ref{countint}.
The second equality in Equation~\eqref{maincount} follows from the first one and the change of variable $\delta=[(e^vc)^{-1}](\gamma)$.
\end{proof}
\subsection{Maximization criterion}\label{secmax}

The following maximization criterion helps to find symmetries $h$ such that $|h(K[\x])\cap K[\x]|$ is maximum among all symmetries with $h=e^{v\x}(c+d\x)$ and $d\in Z(R)$ fixed. It is important to emphasize that it need not find symmetries with maximum values among all symmetries satisfying conditions 1 and 2 of contrapuntal symmetry, but is very useful to discard a number of symmetries whose values are not maximum. 

\begin{theorem}[\textbf{Maximization criterion}]\label{max}  Assume the hypotheses of Theorem~\ref{mainform}. The right-hand side of Equation~\eqref{maincount}
\[\rho\sum\limits_{\gamma\in \Coker(d\cdot-)}|K_{\gamma}||K_{[e^vc](\gamma)}|\]
is maximum, for $e^vc$ ranging over all symmetries of $R$, if and only if $|K_{[e^vc](\gamma)}|=|K_{\gamma}|$ for each $\gamma\in \Coker(d\cdot-)$. In particular, the sum is maximum if $e^vc\equiv e^01\pmod{d}$, that is,\footnote{\textbf{Definition}: we say that $x\equiv y\pmod{d}$ if $[x]=[y]$ in $R/\Ima(d\cdot-)$.} if $v\equiv 0\pmod{d}$ and $c\equiv 1\pmod{d}$. Further, the maximum value is 
\[\rho\sum\limits_{\gamma\in \Coker(d\cdot-)}|K_{\gamma}|^2.\]
\end{theorem}
\begin{proof}
Apply the rearrangement inequality (Theorem~\ref{r.i.}) to the sequences $(|K_{\gamma}|)_{\gamma}$ and $(|K_{[e^vc](\gamma)}|)_{\gamma}$, the latter being a rearrangement of the former by Section~\ref{commufor}.
\end{proof}

Certainly, the criterion is useful because if a symmetry $e^{v\x}(c+d\x)$ satisfies the conditions 1 and 2 for a contrapuntal symmetry and $e^vc\equiv e^01\pmod{d}$, then we get rid of all symmetries with the same $d$ that induce a rearrangement of $(K_{\gamma})_{\gamma}$, which are usually all but those satisfying $e^vc\equiv e^01\pmod{d}$.

\section{The little theorem of first-species counterpoint}\label{little}

The task of finding maximum cardinals subject to the conditions 1 and 2 seems to be difficult in general. In the \textit{Little theorem of counterpoint}, which we prove for the dual numbers case, we establish that each consonance has at least $|K|^2$ admitted successors. For example, in Renaissance counterpoint, each consonance has at least 36 admitted successors. However, this approximation theorem does not provide the admitted successors explicitly, which requires a greater effort.

The proof idea of the little theorem of counterpoint is the following. If we prove that there is at least an $h\in H$ satisfying the conditions 1 and 2 in Theorem~\ref{redu} for all consonances $k\x \in K[\x]$, then we deduce that the number of admissible successors is at least $|h(K[\x])\cap K[\x]|$ by the maximum property of admitted successors sets associated with contrapuntal symmetries. Finally, Theorem~\ref{mainform} gives the exact value of $|h(K[\x])\cap K[\x]|$.

So as to prove the existence of such an $h$, we first need to give concrete criteria for deciding when an $h\in H$ satisfies 1 or 2. We do this in Sections \ref{1con} and \ref{2con}. Finally, we translate our result to any consonance $z+k\x$, by using the bijective translation $e^z$. 
\subsection{The first condition criterion}\label{1con}
Regarding the condition 1 in Theorem~\ref{redu}, given $h=e^{v\x}(c+d\x)$, we have the following equivalences, where $p$ is the polarity $e^ab$ of $(K,D,e^ab)$.
\begin{align*}
k\x \in h(D[\x]) & \Leftrightarrow k \in (c+d\alpha)D+v \\
                             & \Leftrightarrow k \in (c+d\alpha)\cdot p(K)+v\\
                             & \Leftrightarrow v=k-(c+d\alpha)\cdot p(s)\text{ for some }s\in K
\end{align*}

\subsection{The second condition criterion}\label{2con}
As to 2 in Theorem~\ref{redu}, given $h=e^{v\x}(c+d\x)$, we have the following equivalences.
\begin{align*}
p\circ e^v(c+d\alpha)=e^v(c+d\alpha)\circ p& \Leftrightarrow e^ab\circ e^v(c+d\alpha)=e^v(c+d\alpha)\circ e^ab\\
                             & \Leftrightarrow e^{bv+a}b(c+d\alpha)=e^{(c+d\alpha)a+v}(c+d\alpha)b\\
                             & \Leftrightarrow bv+a=(c+d\alpha)a+v\text{ and }b(c+d\alpha)=(c+d\alpha)b
\end{align*}

\subsection{The little theorem}\label{lower}

In the dual numbers case ($\alpha=0$), we can prove the existence of an $h\in H$ with $d=1$ that satisfies 1 and 2 in Theorem~\ref{redu}, for the consonance $k\e$. In fact, according to Sections  \ref{1con} and \ref{2con}, it is enough to show that the following system of equations, 
in the unknowns $v\in R$, $c\in R^*$, $d\in R$ and $s\in K$, has at least a solution.
\[\left\lbrace  \begin{matrix}
v=k-cbs-ca\\
bv+a = ca+v \\
bc = cb
 \end{matrix} \right.\] 
Certainly, it has the solution $v=-ba=a$, $c=b$, $d=1$, $s=k$ (note that $b^2=1$ and $ba+a=0$ since $e^ab$ is involutive by Proposition~\ref{involut}). To conclude, $e^{-ba\e}(b+\e)$ is the desired $h$.

\begin{lemma}\label{prelow}
Let $R$ be a finite ring and $(K,D,p)$ a strong dichotomy. Each consonance $k\e$ has at least $|K|^2$ admitted successors.
\end{lemma}
\begin{proof}
If $h\in H$ satisfies $d=1$, then, according to Lemma~\ref{gencount},
\[
|h(K[\e])\cap K[\e]|=\sum\limits_{r\in R}|(cK+v+r)\cap K|=|cK+v||K|=|K|^2.\]
Hence, $h$, with $h=e^{-ba\e}(b+\e)$, satisfies 1, 2, and the previous equation. Now, let $N$ be the number of admitted successors of $k\e$. If $h'$ satisfies the conditions of Theorem~\ref{redu}, for $k\e$, then \[N\geq |h'(K[\e])\cap K[\e]|\geq |h(K[\e])\cap K[\e]|=|K|^2.\]
\end{proof}

An upper bound for a single contrapuntal symmetry can be obtained as well. First, suppose that $d$ is not $0$. Note that 
\[\sum_{r\in R}|(cK+v+dr)\cap K|=\rho \sum_{r\in \Ima(d\cdot-)}|(cK+v+r)\cap K|,\]
where $d\cdot-:R\longrightarrow R$ is the $R$-endomorphism that sends an element $r\in R$ to $dr$, and $\rho =| \Ker(d\cdot-)|$.

In the case when $e^{v+r}c$ is not the identity $e^01$, note that $|(cK+v+r)\cap K|$ is at most $|K|-1$ since $K$ is rigid. On the other hand, there is at most a value of $r$ in $\Ima(d\cdot-)$ that makes $e^{v+r}c$ the identity, and hence there is at most an $r$ such that $|(cK+v+r)\cap K|=|K|$. Thus, 
\begin{align*}
\rho \sum_{r\in \Ima(d\cdot-)}|(cK+v+r)\cap K| &\leq \rho [(|\Ima(d\cdot-)|-1)(|K|-1)+|K| ]\\
& = \rho \left[\left(\frac{|R|}{\rho}-1\right)(|K|-1)+|K|\right]\\
& =\rho \left[\left(\frac{2|K|}{\rho}-1\right)(|K|-1)+|K|\right]\\
& = 2|K|^2-2|K|+\rho\\
&\leq 2|K|^2-2|K|+|K|=2|K|^2-|K|.
\end{align*}
As to the last inequality note that $\rho$, as a divisor of $|R|$ that is not\footnote{In fact $\rho$, which is by definition $|\Ker(d\cdot-)|$, is equal to $|R|$ if and only if $d=0$--but we are assuming $d\neq 0$.} $|R|$, must be less than or equal to $|K|$ (which coincides with $|R|/2$), the greatest divisor of $|R|$ different from $|R|$.  

Now suppose that $d=0$. So as to find an upper bound for 
\[\sum_{r\in R}|(cK+v)\cap K|\]
we consider two cases. If $c\neq 1$, then $e^vc$ is not the identity $e^01$. If $c=1$, then, since we require $h$ (for $h=e^{v\e}c$) to satisfy condition 1 in Theorem~\ref{redu}, by Section \ref{1con}, $v=k- p(s)$ for some $s\in K$. This means that $v\neq 0$ since $p(s)\in D$, and hence $e^vc$ is not the identity. In both cases, $|(cK+v)\cap K|\leq |K|-1$ since $K$ is rigid. Thus,
\[\sum_{r\in R}|(cK+v)\cap K|\leq |R|(|K|-1)=2|K|(|K|-1).\]

To sum up, collecting the results for $d\neq 0$ and $d=0$ and using Theorem \ref{redu}, we obtain the upper bound in the following theorem.

\begin{theorem}[\textbf{Little theorem of counterpoint}]\label{low}
Let $R$ be a finite ring and $(K,D,p)$ a strong dichotomy. Each consonance $x+k\e$ has at least $|K|^2$  admitted successors, and at most $2|K|^{2}-|K|$ for a single contrapuntal symmetry.
\end{theorem}
\begin{proof}
By Theorem~\ref{redu}, $e^z$ is a bijection between the sets of admitted successors of $k\e$ and $z+k\e$. Thus, the number of admitted successors of $z+k\e$ is equal to the number of admitted successors of $k\e$. The lower bound now follows from Lemma \ref{prelow}. The upper bound corresponds to the preceding discussion.
\end{proof}

\subsection{Steps for the computation of contrapuntal symmetries}
The steps for calculating the contrapuntal symmetries $e^{v\x}(c+d\x)$ in $H$, for a consonance of the form $k\x$, are the following. We start with a strong dichotomy $(K,D,e^ab)$.
\begin{enumerate}
\item Solve the following system of equations in the unknowns $c\in R^*$, $v$, and $d$.
\[\left\lbrace  \begin{matrix}
bv+a = (c+d\alpha)a+v \\
b(c+d\alpha)=(c+d\alpha)b\\
 \end{matrix} \right.\]
 This corresponds to condition 2 of contrapuntal symmetry (Section \ref{2con}).
\item For each consonance $k\x$, among the symmetries of the form $e^{v\x}(c+d\x)$ obtained in 1, choose those with $v\in k-cD$. This corresponds to condition 1 of contrapuntal symmetry (Section \ref{1con}). The reason for first computing the symmetries satisfying condition 2 is that those symmetries are usually less than those satisfying condition 1, so we perform less operations.
\item For each consonance $k\x$, among the symmetries $e^{v\x}(c+d\x)$ obtained in 2, choose those such that the right-hand term of Equation~\eqref{forgen}  
\[\rho\sum\limits_{\gamma\in \Coker(d\cdot-)}|(cK+v)_{\gamma}||K_{\gamma}|\]
with $K'=cK+v$ is maximum. This ensures that the condition 3 holds by Equation~\eqref{counting}. In the case when $R$ is commutative, maximize the right-hand side of Equation~\eqref{maincount}. In the case when $R=\mathbb{Z}_n$, maximize Equation~\eqref{maincountforint}. So as to discard a number of symmetries, if $R$ is commutative, Theorem~\ref{max} can be used.
\end{enumerate} 

In the next section we apply the previous steps to computing the contrapuntal symmetries in the case of first-species Renaissance counterpoint.
\section{First-species Renaissance counterpoint}\label{hand}

Let $(K,D,e^2 5)$ be the Renaissance strong dichotomy of $\mathbb{Z}_{12}$, with $K=\{0,3,4,7,8,9\}$ and $D=\{1,2,5,6,10,11\}$. Next, we compute the respective contrapuntal symmetries and admitted successors, in the dual numbers case ($\alpha=0$). In Section~\ref{1var}, we explain a simple procedure for obtaining the contrapuntal symmetries in the case $\alpha=1$.

\subsection{Condition 2}\label{secondcond}

By the commutativity of $\mathbb{Z}_{12}$, the condition 2 for symmetries $e^{v\e}(c+d\e)\in H$, $c\in R^*$, reduces to solve the equation
\[5v+2=c2+v,\]
which is equivalent to
\[4v=2(c-1).\]
If $c=1,7$, then the equation becomes $4v=0$, so $v=0,3,6,9$. If $c=5,11$, then the equation becomes $4(v-2)=0$ and hence $v=2,5, 8,11$. The solutions are summarized in the following table.
\begin{center}
\begin{tabular}{|l|l|}
\hline
$c$ & $v$  \\ \hline
 1, 7 & 0, 3, 6, 9 \\ \hline
 5, 11 & 2, 5, 8, 11 \\ \hline
\end{tabular}
\end{center}

\subsection{Condition 1}\label{firstcond}

Now, among these solutions we choose, for each consonance $k\e$, those satisfying 
\[v\in k-cD.\] 
In the following table, we organize the results of the operations involved. Specifically, we compute $3\mathbb{Z}_{12}\cap (k-cD)$ for $c=1,7$ and $(2+3\mathbb{Z}_{12})\cap (k-cD)$ for $c=5,11$.

\begin{center}
\begin{tabular}{|l|l|l|l|l|l|l|}
\hline

\backslashbox{$c$}{$k$}& 0 &  3 &4 & 7& 8& 9\\ \hline
1 & $\{6\}$ &$\{9\}$ &$\{3,6\}$  &$\{6,9\}$ & $\{3,6,9\}$ & $\{3\}$\\ \hline
5&  $\{2,5,11\}$ & $\{2,5,8\}$ &$\{2,11\}$&$\{2,5\}$&$\{2\}$&$\{2,8,11\}$\\ \hline
 7 & $\{6\}$ &$\{9\}$ &$\{6,9\}$  &$\{0,9\}$ & $\{3,6,9\}$ & $\{3\}$ \\ \hline
 11 &$\{2,5,11\}$ & $\{2,5,8\}$ &$\{2,5\}$&$\{5,8\}$&$\{2\}$&$\{2,8,11\}$\\ \hline
\end{tabular}
\end{center}

We can easily fill in the table as follows. There are six entries, namely those labelled by $(c,k)$ with $c=1,5$ and $k=0,4,8$, which we can start from, the other being obtained by using certain symmetries. In fact, note that \[7(3\mathbb{Z}_{12}\cap (k-D))=3\mathbb{Z}_{12}\cap (7k-7D)),\] \[7((2+3\mathbb{Z}_{12})\cap (k-5D))=(2+3\mathbb{Z}_{12})\cap (7k-11D),\]
 \[\pm 3+(3\mathbb{Z}_{12}\cap (k-cD))=3\mathbb{Z}_{12}\cap (\pm 3+k-cD),\] 
 and
\[\pm 3+((2+3\mathbb{Z}_{12})\cap (k-cD))=(2+3\mathbb{Z}_{12})\cap (\pm 3+k-cD).\]
For example, the entry $(1,3)$ is obtained by adding $3$ to the entry $(1,0)$, and the entry $(7,0)$ is obtained by multiplying the entry $(1,0)$ by $7$.
\subsection{Maximization}\label{countsym}

In the case of the Renaissance dichotomy, Theorem~\ref{max} allows us to obtain the following table of maximum values, without the restrictions of conditions 1 and 2 of contrapuntal symmetry, and their corresponding symmetries. We do not include the value $d=0$, whose maximum is not useful because the identity does not satisfy the condition 1. After the table we justify the results.

\begin{center}
\begin{tabular}{|l|l|c|l|}
\hline
$d$ & $\rho$ &  maximum $\sum\limits_{r=0}^{11}|(cK+v+dr)\cap K|$ & $e^vc$\\ \hline
 1, 5, 7, 11 & 1 & 36 & any\\ \hline
 2, 10 & 2 & 36& any \\ \hline
 3, 9 & 3 & 42 & $\equiv e^01\pmod{3}$\\ \hline
 4, 8 & 4 & 56 & $\equiv e^01\pmod{4}$\\ \hline
 6 & 6 & 48  & $\equiv e^01\pmod{6}$ \\\hline
\end{tabular}
\end{center}

If $\rho=3$, since all entries of the vector $(|K_0|,|K_1|,|K_2|)$ are different (Example~\ref{vectors}), then the unique rearrangement that coincides with it is the composition with $e^01$, and hence the maximum is only taken for the identity modulo $3$. The same is true for $\rho=4$. If $\rho=6$, then $(|K_0|,|K_1|,|K_2|,|K_3|,|K_4|,|K_5|)=(1,1,1,2,1,0)$. Now, if a symmetry $e^vc$ modulo 6 induces a rearrangement that leaves the vector invariant, then the rearrangement leaves $K_3$ invariant, but $e^vc(3)=3+v=3$ and hence $v=0$. Moreover, $e^05(1)=5$ and hence it does not induce a rearrangement that leaves $K_1$ invariant, so $e^vc=e^01$.  
\begin{center}
\textbf{Contrapuntal symmetries for $0\e$.}
\end{center}

According to the table in Section~\ref{firstcond}, there are exactly two symmetries for $\rho=6$, namely $e^{6\e}(1+6\e)$ and $e^{6\e}(7+6\e)$, with $e^vc$ congruent to $e^01$ modulo $6$. This allows us to discard the cases $\rho=3,2,1$. Up to now the maximum is $48$, and it remains to examine the cases $\rho=4,12$.

In the case when $\rho=4$, the residues modulo 4 of the candidates in the table produce the following sums, calculated with Equation~\eqref{maincountforint}. 

\begin{center}
\begin{tabular}{|l|l|l|l|}
\hline
\backslashbox{$c$}{$v$}& 1 &  2 & 3 \\ \hline
1 & 36 & 16 & 36  \\ \hline
3&  24 & 20 & 48\\ \hline
\end{tabular}
\end{center}

Thus, we have find two new symmetries with sum 48 (the maximum up to now), namely $e^{11 \e}(11+8\e)$ and $e^{11 \e}(11+4\e)$. Here, $e^{11}11 \equiv e^33\pmod{4}$.

It remains to examine the case $\rho=12$. We have the following sums.
\begin{center}
\begin{tabular}{|l|l|l|l|l|}
\hline
\backslashbox{$c$}{$v$}& 6 & 2 & 5 & 11 \\ \hline
1 & 24 &  &  & \\ \hline
7 &  36&  &  &\\ \hline
5 & & 0 &  36&  36\\ \hline
11&  & 12 & 24 & 48\\ \hline
\end{tabular}
\end{center}

This means that the maximum sum is $48$ and that there is yet another symmetry $e^{11\e}11$.

\begin{center}
\textbf{Contrapuntal symmetries for $3\e$}
\end{center}
According to the second table in Section~\ref{firstcond}, there are exactly two symmetries for $\rho=4$, namely $e^{8\e}(5+4\e)$ and $e^{8\e}(5+8\e)$, with $e^vc$ congruent to $e^01$ modulo $4$. We thus discard the cases $\rho=6,3,2,1$. Up to now the maximum is $56$.

The remaining case is $\rho=12$. We have the following sums.

\begin{center}
\begin{tabular}{|l|l|l|l|l|}
\hline
\backslashbox{$c$}{$v$}& 9 & 2 & 5 & 8 \\ \hline
1 & 36 &  &  & \\ \hline
7 &  24&  &  &\\ \hline
5 & & 0 &  36&  48\\ \hline
11&  & 12 & 24 & 36\\ \hline
\end{tabular}
\end{center}

The maximum is 56 and there are no more contrapuntal symmetries.

\begin{center}
\textbf{Contrapuntal symmetries for $4\e$}
\end{center}
There are exactly two symmetries for $\rho=6$, namely $e^{6\e}(1+6\e)$ and $e^{6\e}(7+6\e)$, with $e^vc$ congruent to $e^01$ modulo $6$. This allows us to discard the cases $\rho=3,2,1$. Up to now the maximum is $48$. The sums are less than 48 for $\rho=4$ as shown in the following table.
\begin{center}
\begin{tabular}{|l|l|l|l|}
\hline
\backslashbox{$c$}{$v$}& 1 &  2 & 3 \\ \hline
1 &  & 16 & 36  \\ \hline
3&  24 & 20 & \\ \hline
\end{tabular}
\end{center}
The case $\rho=12$ yields the following sums.

\begin{center}
\begin{tabular}{|l|l|l|l|l|l|l|}
\hline
\backslashbox{$c$}{$v$}& 3 & 6 & 2 & 11 & 9 & 5 \\ \hline
1 & 36 & 24  &  &  & &\\ \hline
5 &  &  & 0 &36 & &\\ \hline
7 & & 36 &  &   &24&\\ \hline
11&  &  & 12 & &&24\\ \hline
\end{tabular}
\end{center}
Hence, the maximum is 48 and there are no more contrapuntal symmetries.

\begin{center}
\textbf{Contrapuntal symmetries for $7\e$}
\end{center}
The case $\rho=12$ yields the following sums.

\begin{center}
\begin{tabular}{|l|l|l|l|l|l|l|}
\hline
\backslashbox{$c$}{$v$}& 0 & 6 & 2 & 8 & 9 & 5 \\ \hline
1 &  & 24  &  &  & 36&\\ \hline
5 &  &  & 0 &  & &36\\ \hline
7 &60 &  &  &   &24&\\ \hline
11&  &  &  & 36&&24\\ \hline
\end{tabular}
\end{center}
Thus, we discard all remaining cases for $\rho$ and the maximum is 60, corresponding to a unique contrapuntal symmetry $e^07$.

\begin{center}
\textbf{Contrapuntal symmetries for $8\e$}
\end{center}
There are exactly two symmetries for $\rho=6$, namely $e^{6\e}(1+6\e)$ and $e^{6\e}(7+6\e)$, with $e^vc$ congruent to $e^01$ modulo $6$. This allows us to discard the cases $\rho=3,2,1$. Up to now the maximum is $48$. 

The sums for $\rho=4$ are the same of the consonance $0\e$. In the table from Section~\ref{firstcond}, we observe that there are two additional symmetries, namely $e^{3\e}(7+4\e)$ and $e^{3\e}(7+8\e)$, with $e^vc$ congruent to $e^33$ modulo $4$.

The case $\rho=12$ yields the following sums.

\begin{center}
\begin{tabular}{|l|l|l|l|l|}
\hline
\backslashbox{$c$}{$v$}& 2 & 3 & 6 & 9 \\ \hline
1 &  &  36 & 24 &36  \\ \hline
5 & 0 &  &  & \\ \hline
7 & & 48 & 36 & 24  \\ \hline
11& 12 &  &  & \\ \hline
\end{tabular}
\end{center}
Hence, the maximum is 48 and there is an additional contrapuntal symmetry $e^{3\e}7$.

\begin{center}
\textbf{Contrapuntal symmetries for $9\e$}
\end{center}
There are exactly two symmetries for $\rho=4$, namely $e^{8\e}(5+4\e)$ and $e^{8\e}(5+8\e)$, with $e^vc$ congruent to $e^01$ modulo $4$. We thus discard the cases $\rho=6,3,2,1$. Up to now the maximum is $56$.

For $\rho=12$ we have the following sums.

\begin{center}
\begin{tabular}{|l|l|l|l|l|}
\hline
\backslashbox{$c$}{$v$}& 3 & 2 & 11 & 8 \\ \hline
1 & 36 &  &  & \\ \hline
7 &  48&  &  &\\ \hline
5 & & 0 &  36&  48\\ \hline
11&  & 12 & 48 & 36\\ \hline
\end{tabular}
\end{center}

The maximum is 56 and there are no more contrapuntal symmetries.

\subsection{Admitted successors}\label{secsuc}

We can finally obtain the list of admitted successors, by directly using Equation~\eqref{presuc} for each contrapuntal symmetry $h$ obtained in Section~\ref{countsym}; see Table \ref{tab:adsu}.   
 
The \textbf{prohibition of parallel fifths} is an important conclusion that we can draw from the table. In fact, the consonance $7\e$ has as set of admitted successors $\mathbb{Z}_{12}+(K\setminus\{7\})\e$, that is, any interval can follow a fifth, except a fifth.

\begin{table}
\begin{center}
\begin{tabular}{|c|c|c|c|}
\hline
$k$ & $|h(K[\e])\cap K[\e]|$ & $h$ & admitted successors of $k\e$ \\ \hline
\multirow{11}{*}{$0$} & \multirow{11}{*}{$48$} & \multirow{2}{*}{ $e^{6\e}(1+6\e)$}& $r+\{3,9\}\e$, $r$ even \\     
 &  &  & $r+K\e$, $r$ odd  \\ \cline{3-4}
 & & \multirow{2}{*}{ $e^{6\e}(7+6\e)$} & $r+\{3,7,9\}\e$, $r$ even  \\ 
 &  &  & $r+(K\setminus\{7\})\e$, $r$ odd  \\ \cline{3-4}
 &  & \multirow{3}{*}{ $e^{11\e}(11+8\e)$}&  $\{0,3,6,9\}+\{3,4,7,8\}\e$\\ 
 &  &  &  $\{1,4,7,10\}+\{0,3,7,8\}\e$ \\ 
  &  &  & $\{2,5,8,11\}+\{0,3,4,7\}\e$ \\ \cline{3-4}
 &  & \multirow{3}{*}{ $e^{11\e}(11+4\e)$}&  $\{0,3,6,9\}+\{3,4,7,8\}\e$\\ 
 &  &  &  $\{1,4,7,10\}+\{0,3,4,7\}\e$ \\ 
  &  &  & $\{2,5,8,11\}+\{0,3,7,8\}\e$ \\ \cline{3-4}
  &  &   $e^{11\e}11$ & $\mathbb{Z}_{12}+\{3,4,7,8\}\e$ \\ \hline
\multirow{6}{*}{$3$}& \multirow{6}{*}{$56$} & \multirow{3}{*}{ $e^{8\e}(5+8\e)$}&  $\{0,3,6,9\}+\{0,4,7,8\}\e$\\ 
 &  &  &  $\{1,4,7,10\}+(K\setminus\{7\})\e$ \\ 
  &  &  & $\{2,5,8,11\}+(K\setminus\{9\})\e$ \\ \cline{3-4}
 &  & \multirow{3}{*}{ $e^{8\e}(5+4\e)$}&  $\{0,3,6,9\}+\{0,4,7,8\}\e$\\ 
 &  &  &  $\{1,4,7,10\}+(K\setminus\{9\})\e$ \\ 
  &  &  & $\{2,5,8,11\}+(K\setminus\{7\})\e$ \\ \hline
\multirow{2}{*}{$4$}& \multirow{2}{*}{$48$} & $e^{6\e}(1+6\e)$&  \multirow{2}{*}{see $k=0$}\\ 
                    &  & $e^{6\e}(7+6\e)$ &  \\ \hline
$7$ & $60$ & $e^07$ & $\mathbb{Z}_{12}+(K\setminus\{7\})\e$ \\ \hline
\multirow{9}{*}{$8$} & \multirow{9}{*}{$48$} & $e^{3\e}7$& $\mathbb{Z}_{12}+\{0,3,4,7\}\e$\\ \cline{3-4}
                     & & $e^{6\e}(1+6\e)$&  \multirow{2}{*}{see $k=0$}\\ 
                    &  & $e^{6\e}(7+6\e)$ &  \\ \cline{3-4}
 &  & \multirow{3}{*}{ $e^{3\e}(7+4\e)$}&  $\{0,3,6,9\}+\{0,3,4,7\}\e$\\ 
 &  &  &  $\{1,4,7,10\}+\{3,4,7,8\}\e$ \\ 
  &  &  & $\{2,5,8,11\}+\{0,3,7,8\}\e$ \\ \cline{3-4}
 &  & \multirow{3}{*}{ $e^{3\e}(7+8\e)$}&  $\{0,3,6,9\}+\{0,3,4,7\}\e$\\ 
 &  &  &  $\{1,4,7,10\}+\{0,3,7,8\}\e$ \\ 
  &  &  & $\{2,5,8,11\}+\{3,4,7,8\}\e$ \\ \hline
\multirow{2}{*}{$9$}& \multirow{2}{*}{$56$} & $e^{8\e}(5+8\e)$&  \multirow{2}{*}{see $k=3$}\\ 
                    &  & $e^{8\e}(5+4\e)$ &  \\ \hline
\end{tabular}
\end{center} 
\caption{Contrapuntal symmetries and admitted successors for the cantus firmus $0$. We obtain the admitted successors of $z+k\e$ by adding $z$ to the cantus firmus of the results.}
\label{tab:adsu}
\end{table}

\section{A non-commutative counterpoint}\label{handnc}
Consider the data of Example \ref{noncomex}. There, we have the strong dichotomy $(K,D,e^I)$ of the ring $R$ of upper triangular matrices on $\mathbb{Z}_2$, with $K=\{\mathbf{0},A_1,A_2,A_3\}$ and $D=\{I,B_1,B_2,B_3\}$, where $B_i=A_i+I$ for $i=1,2,3$. In this case, $b=I$ and $a=I$. Let us compute the contrapuntal symmetries and admitted successors in the dual numbers case.

\subsection{Condition 2}\label{secondcondnc}

We solve the equation
\[v+I=c+v,\]
so $c=I$, and $v$ and $d$ range over $R$.
\subsection{Condition 1}\label{firstcondnc}

Now, among the previous solutions, we take, for each consonance $k\e$, those satisfying
\[v\in k-D=k+D.\] 
The following table contains the results. 

\begin{center}
\begin{tabular}{|l|l|l|l|l|l|l|}
\hline
$k$& $\mathbf{0}$ &  $A_1$ & $A_2$ & $A_3$\\ \hline
$k+D$ & $D$ &$\{B_1,I,A_3,A_2\}$ &$\{B_2,A_3,I,A_1\}$  &$\{B_3,A_2,A_1,I\}$\\ \hline
\end{tabular}
\end{center}

\subsection{Maximization}\label{countsymnc}

Since $c=I$, and $I\in Z(R)$, the maximization criterion (Theorem \ref{max}) remains valid in this case and we use it. 

We start by computing $\Ima (r\cdot-)$, $\rho$, $K/\Ima (r\cdot-)$, and $(|K_{\gamma}|_{\gamma})$ in the following table. The ordering on $\gamma$ is that written for the classes in $K/\Ima (r\cdot-)$. As before, the maximization criterion is not useful for the case $d=\mathbf{0}$ since the identity does not satisfy the condition 1.

\begin{center}
\begin{tabular}{|l|l|c|l|l|}
\hline
$r$ & $rR$ & $\rho$ &  $ K/\Ima (r\cdot -)$ & $(|K_{\gamma}|_{\gamma})$ \\ \hline
  $A_1,A_2$ & $\{\mathbf{0},A_1,A_2,B_3\}$ & 2 & $\{A_1R,A_1R+A_3\}$ & $(3,1)$ \\ \hline
$A_3,I$ & $R$& 1& $\{R\}$ & $(4)$\\ \hline
$B_i$ for $i=1,2,3$ & $\{\mathbf{0},B_i\}$ & 4& $\{\{\mathbf{0},B_i\}, \{A_i,I\},\{A_1,A_2,A_3\}\setminus \{A_i\}\}$ & $(1,1,2,0)$\\ \hline

 \end{tabular}
\end{center}

We have the following maximum values with the restriction $c=I$. In the case $\rho=2$, the unique $[e^v]$ inducing the identity rearrangement of $(3,1)$ is the identity modulo $A_1$. In the case $d=B_1$, if $[e^v]$ induces an identity rearrangement of $(1,1,2,0)$, then necessarily $[e^v]$ leaves invariant $K_{\{A_2,A_3\}}$, which is equal to $\{A_2,A_3\}$. The possibilities are either $v=\mathbf{0}$ or $v=B_1$, that is, $e^v$ is the identity modulo $B_1$. The cases $d=B_2$ and $d=B_3$ are similar.

\begin{center}
\begin{tabular}{|l|l|c|l|}
\hline
$d$ & $\rho$ &  maximum $\sum\limits_{r\in R}|(K+v+dr)\cap K|$ & $v$\\ \hline
 $A_1,A_2$ & 2 & 20 & $\mathbf{0},A_1,A_2,B_3$\\ \hline
 $A_3,I$ & 1 & 16& any \\ \hline
 $B_1$ & 4 & 24 & $\mathbf{0},B_1$\\ \hline
 $B_2$ & 4 & 24 & $\mathbf{0},B_2$\\ \hline
 $B_3$ & 4 & 24  & $\mathbf{0},B_3$ \\\hline
\end{tabular}
\end{center}

Now we must compute the cardinals for $d=\mathbf{0}$, that is, those of the form $8|(K+v)\cap K|$ with $v\neq \mathbf{0}$ by the table in Section \ref{firstcondnc}. We already have the computation of $K+I+k$, which is just $D+k$, for $k\in K$ by the same table. Moreover, $K+A_i=\{\mathbf{0},A_i\}\cup (\{B_1,B_2,B_3\}\setminus \{B_i\})$. Thus, $8|(K+v)\cap K|=0$ if $v=I$ and $8|(K+v)\cap K|= 16$ if $v\notin \{I, \mathbf{0}\}$, so the maximum cardinals (subject to conditions 1 and 2) occur for the symmetries with maximum value $24$ in the preceding table. The following table shows the contrapuntal symmetries directly obtained. 

\begin{center}
\begin{tabular}{|l|l|c|l|}
\hline
$k$            & $h$                                        \\ \hline
$\mathbf{0}$   & $ e^{B_i\e}(I+B_i\e)$ for $i=1,2,3$ \\ \hline
$A_i$ (with $i=1,2,3$)   & $e^{B_i\e}(I+B_i\e)$ \\ \hline
\end{tabular}
\end{center}

\subsection{Admitted successors}\label{secsucnc}

The direct use of Equation~\eqref{presuc} for each contrapuntal symmetry $h$ obtained in Section~\ref{countsymnc} produces the following table. In this example, the maximization criterion is useful to determine all contrapuntal symmetries. 

\begin{center}
\begin{tabular}{|c|c|c|c|}
\hline
$k$                                      & $h$                                      & admitted successors of $k\e$ \\ \hline
$\mathbf{0}$ &  $e^{B_i\e}(I+B_i\e)$, $1\leq i\leq 3$& see $k=A_i$\\ \hline
\multirow{2}{*}{$A_i$ (with $i=1,2,3$)}  & \multirow{2}{*}{$e^{B_i\e}(I+B_i\e)$} & $\{A_3,I,B_1,B_2\}+K\e$ \\ 
                                         &                          & $\{\mathbf{0},A_1,A_2,B_3\}+(\{A_1,A_2,A_3\}\setminus \{A_i\})\e$    \\ \hline
\end{tabular}
\end{center} 

Note that some parallelisms of the consonance $A_i\e$ are forbidden in this notion of counterpoint. The counterpoint notion that results from multiplication on the right is very similar.\footnote{To be more exact, this notion is obtained by applying the previous theory to the opposite ring $R^{op}$ instead of $R$. In that case, all right $R$-modules considered become left $R$-modules.}

\section{Local characterization equivalences}\label{loccharequiv}

In this section we compare our condition 2 in Definition~\ref{def} with Mazzola's original condition \cite[Definition~95]{MazzolaTopos}. 

The following result inspires the latter. It is a local version of the uniqueness property of a strong dichotomy. Observe that the conjugation in Proposition~\ref{conj} restricts to bijections $H(K[\x],D[\x])\longrightarrow H_z(K[\x],D[\x])$ between respective sets of symmetries sending $K[\x]$ to $D[\x]$ (check). 

\begin{theorem}\label{indxdichori}
Let $(K,D,e^ab)$ be a strong dichotomy of $R$. For each $z\in R$, $p^z[\x]$ is the unique $z$-invariant symmetry such that $(K[\x],D[\x],p^z[\x])$ is a self-complementary dichotomy of $R[\x]$. In particular, by Lemma~\ref{fiber}, 
\begin{equation*}
p^z[\x](z+K\x)=z+D\x.
\end{equation*}
\end{theorem}
\begin{proof}
Since self-complementary dichotomies $(K[\x],D[\x],p')$ with $p'$ $z$-invariant are in correspondence (conjugation) with self-complementary dichotomies $(K[\x],D[\x],p)$ where $p$ is $0$-invariant, the proof of the theorem reduces to the case when $z=0$ (Proposition~\ref{unique0ori}). There, $p^0[x]=e^{a\x}b$, so $p^z[\x]=e^z\circ e^{a\x}b\circ e^{-z}=e^{(1-b)z+ a\x}b$.
\end{proof}

\begin{proposition}\label{unique0ori}
Let $(K,D,e^ab)$ be a strong dichotomy of $R$. The symmetry $e^{a\x}b$ is the unique $0$-invariant symmetry such that $(K[\x],D[\x],e^{a\x}b)$ is a self-complementary dichotomy of $R[\x]$. 
\end{proposition}
\begin{proof}
Let $h$ be a $0$-invariant symmetry, where $h=e^{u+v\x}(c+d\x)$. By Proposition~\ref{charinv}, $u=0$. Moreover, if $h(K[\x])=D[\x]$, from Lemma~\ref{fiber}, we obtain that $(c+d\alpha)K+v=D$, so the condition that $e^ab$ is a polarity implies that $c+d\alpha=b$ and $v=a$. Further, we claim that $d=0$. In fact, by the condition that $h(K[\x])=D[\x]$,
\[e^{v\x}(c+d\x)(1+K\x)=c+((c+d\alpha)K+v+d)\x=c+(D+d)\x\subseteq D[\x]\]
and 
\[e^{v\x}(c+d\x)(-1+K\x)=-c+((c+d\alpha)K+v-d)\x=-c+(D-d)\x\subseteq D[\x],\]
so $D+d\subseteq D$ and $D-d\subseteq D$, and hence $D+d\subseteq D$ and $D\subseteq D+d$. Thus, $D=D+d$ and, since both $K$ and $D$ are rigid (Corollary~\ref{rig}), $d=0$. 

Up to now, we know that each $0$-invariant symmetry $h$ such that $h(K[\x])=D[\x]$ is necessarily $e^{a\x}b$. It remains to show that $(K[\x],D[\x],e^{a\x}b)$ is a self-complementary dichotomy. This was done in Proposition~\ref{dualexts}.
\end{proof}

\begin{lemma}\label{fiber} If $h$ is a $z$-invariant symmetry of $R[\x]$ such that $h(K[\x])=D[\x]$, then $h(z+K\x)=z+D\x$.
\end{lemma}
\begin{proof} Under our assumptions,
\[h(z+K\x)=h(K[\x]\cap(z+R\x))=h(K[\x])\cap(z+R\x)= D[\x]\cap (z+R\x)=z+D\x.\]
\end{proof}  

We have the following corollary of Theorem~\ref{indxdichori}.

\begin{corollary}\label{extrig}
Let $(K,D,e^ab)$ be a strong dichotomy of $R$. For each $z\in R$, $K[\x]$ is $H_z$-rigid, and $f(K[\x])=f'(K[\x])$ implies $f=f'$ whenever $f,f'\in H_z$.
\end{corollary}
\begin{proof}
As in the proof of Proposition~\ref{corstab}, there is a bijection between $H_z(K[\x],D[\x])$ and the stabilizer $\theta(K[\x])$ with respect to $H_z$. Thus, Theorem~\ref{indxdichori} implies that $\theta(K[\x])$ is trivial. 

Now, suppose that $f(K[\x])=f'(K[\x])$ with $f,f'$ in $H_z$. This implies that $f'^{-1}\circ f(K[\x])=K[\x]$, so $f'^{-1}\circ f=id$, and $f=f'$.
\end{proof}

If we apply the condition of Theorem~\ref{indxdichori} to a dichotomy $\{\tilde{K},\tilde{D}\}$ of $R[\x]$, we obtain the following one.
\begin{itemize}
\item The symmetry $p^z[\x]$ is the unique $p'\in H_z$ such that $(\tilde{K},\tilde{D},p')$ is a self-complementary dichotomy.
\end{itemize}
As explained in \cite[Section~14.2]{Beau}, it is a sort of characteristic   property of the partition $\{z+K\x,z+D\x\}$ in the case of the Renaissance dichotomy $(K,D,p)$. 

If we apply it to a deformed dichotomy $\{g(K[x]),g(D[x])\}$, for $g\in \sym(R[\x])$, we get the following condition.
\begin{itemize}
\item The symmetry $p^z[\x]$ is the unique $p'\in H_z$ such that $(g(K[x]),g(D[x]),p')$ is a self-complementary dichotomy.
\end{itemize}
First, let us prove that the latter is just Mazzola's condition 2.

\begin{proposition}\label{equivmazzcond}
Let $(K,D,p)$ be a strong dichotomy, $z\in R$, and $g\in \sym(R[\x])$. The following conditions are equivalent.\begin{enumerate} 
\item The symmetry $p^z[\x]$ is the unique $p'\in H_z$ such that $(g(K[\x]),g(D[\x]),p')$ is a self-complementary dichotomy.
\item The triple $(g(K[x]),g(D[x]),p^z[\x])$ is a self-complementary dichotomy.
\item The equation $p^z[\x]g_z=g_z p^z[\x]$ holds, where $g_z$ is as in Propostion~\ref{hache}.
\end{enumerate}
\end{proposition}
\begin{proof}
For each $p'\in H_z$, by Proposition~\ref{hache} and Corollary~\ref{extrig}, we have the following equivalences.
\begin{align*}
p'(g(K[\x])) =g(D[\x]) &\Leftrightarrow \\ p'(g_z(K[\x])) =g_z(D[\x]) &\Leftrightarrow \\ p'g_z(K[\x]) =g_zp^z[\x](K[\x])& \Leftrightarrow \\ p'g_z =g_zp^z[\x]
\end{align*}
The last (and hence the first) equation has a unique solution for the unknown $p'$, so 1 is equivalent to 3 and 2. 
\end{proof}

The condition of Proposition~\ref{equivmazzcond} is stronger than our characterizing condition in Definition~\ref{def}. In fact, if $p^z[\x](g(K[\x]))=g(D[\x])$, then 
\[p^z[\x](g(K[\x])\cap I_z)=p^z[\x](g(K[\x]))\cap p^z[\x](I_z)=g(D[\x])\cap I_z.\]
Finally, we establish that under the commutativity of $R$, both conditions are equivalent.

\begin{theorem}\label{equivcommu}
Let $(K,D,p)$, with $p=e^ab$, be a strong dichotomy of a commutative ring, $z\in R$, and $g\in \sym(R[\x])$. The following equations are equivalent.
\begin{enumerate} 
\item $p^z[\x](g(K[\x]))=g(D[\x])$
\item $p^z[\x](g(K[\x])\cap I_z)=p^z[\x](g(D[\x]))\cap I_z$
\end{enumerate}
\end{theorem}
\begin{proof}
Suppose that $g=e^{u+v\x}(c+d\x)$ and write $w=v+dc^{-1}(z-u)$. By unraveling the condition~3 in Proposition~\ref{equivmazzcond}, which is just 1, it is equivalent to the following set of equations.
\[\left\lbrace  \begin{matrix}
bw+a = (c+d\alpha)a+w\\
bc = cb\\
bd = db\\
 \end{matrix} \right.\]
On the other hand, 2 is equivalent to the condition~3 in Theorem~\ref{equivcond}, that is, to the following system.
\[\left\lbrace  \begin{matrix}
bw+a = (c+d\alpha)a+w\\
b(c+d\alpha) = (c+d\alpha)b\\
\end{matrix} \right.\]
Hence, 1 and 2 are equivalent when $R$ is commutative.
\end{proof}

\section{Variations of the model}\label{vars}

The validity of Definition~\ref{def} for the case $\alpha=1$ of the contrapuntal intervals ring, and the motivations in Sections~\ref{alternation}, \ref{locchar}, and \ref{variety}, give rise to the following questions on the classical model.
\begin{itemize}
\item Are dual numbers ($\alpha=0$) the most appropriate structure to model contrapuntal intervals? According to Mazzola, the discantus is a tangential alteration of the cantus firmus, so this principle leads to dual numbers, which are the natural tangents in algebraic geometry. However, we could ask why this kind of alteration is involved in counterpoint and not another one.

\item Why do we not require the local characterization condition on $\{g(K[\x]),g(D[\x])\}$ for the cantus firmus of a possible successor $\eta$?

\item More radically, why do we not require the local characterization on $\{g(K[\x]),g(D[\x])\}$ for all fibers?

\end{itemize} 

These questions lead to the following three variations of the model. We indicate how to compute the contrapuntal symmetries for the Renaissance dichotomy in each case. The musicological analysis of these results is in \cite{Musicology}.

\subsection{First variation}\label{1var}
This variation corresponds to Definition~\ref{def} with $\alpha=1$. 

First, let us note that we can relate the models for $\alpha=1$ and $\alpha=0$. Certainly, there is an injection $\phi:e^{v\x}(c+d\x)\mapsto e^{v\e}((c+d)+d\e)$ from the set of symmetries in $H$ satisfying\footnote{See Sections~\ref{1con} and \ref{2con} for explicit forms of these conditions.} 1 and 2 in Theorem~\ref{redu} for $\alpha=1$ to the set of symmetries in $H$ satisfying 1 and 2 for $\alpha=0$. Moreover, $\phi$ preserves the cardinality of successors sets. In fact, if $h=e^{v\x}(c+d\x)$, then 
\[|h(K[\x])\cap K[\x]|=\sum\limits_{r\in R}|((c+d)K+v+dr)\cap K|= |\phi(h)(K[\e])\cap K[\e]|.\]

Thus, if $M_0$ and $M_1$ are the maximum cardinalities according 3 in Theorem~\ref{redu} for $\alpha=0$ and $\alpha=1$, respectively, then $M_0\geq M_1$.

Moreover, if there is a contrapuntal symmetry in the image of $\phi$, like in Table~\ref{tab:adsu}, then $M_0=M_1$. In this case, we can compute the contrapuntal symmetries for $\alpha=1$ as all inverse images of contrapuntal symmetries for $\alpha=0$.

Also, note that $e^{v\e}(c+d\e)$ is in the image of $\phi$ if and only if $c-d\in R^*$, in which case its inverse image is $e^{v\x}((c-d)+d\x)$.

\subsection{Second variation}\label{2var}

If we also require the local characterization on the fibers of the possible successors of a consonance, then we change condition 3 in Definition~\ref{def} for the following one, where $P(z)$ is the property 2 in Definition~\ref{def}.
\begin{itemize}
\item The cardinality of $\{z'+k'\x \in g(K[\x])\cap K[\x]\ |\ P(z')\}$ is maximum among all $g$ satisfying 1 and 2. 
\end{itemize}

By following the same steps to prove Theorem~\ref{redu}, we obtain (check) the following simplification of this definition.

\begin{theorem}\label{redu2} The admitted successors of $k\x\in K[\x]$ are the elements of the sets $T$ of the form
\[\{z+k\x\in h(K[\x])\cap K[\x])\ |\ P(z)\},\]
where $h=e^{v\x}(c+d\x)\in H$ and 
\begin{itemize}
\item[(a2)] $k\x\in h(D[\x])$, 
\item[(b2)] $p\circ e^v(c+d\alpha)=e^v(c+d\alpha)\circ p$, and 
\item[(c2)] the cardinality of $T$ is maximum among all $h\in H$ satisfying (a2) and (b2). 
\end{itemize}
\end{theorem}
Suppose that $h=e^{v\x}(c+d\x)\in H$. Note that once (b2), which is just $P(0)$, holds, the proof of Theorem~\ref{equivcommu} implies that $P(z)$ is equivalent to $bdc^{-1}z=dc^{-1}z$. In turn, $P(cr)$ is equivalent to $bdr=dr$, and the latter is equivalent to $P((c+d)r)$ whenever $(c+d)\in R^*$ and $R$ is commutative.

Asumme that $R$ is commutative. We relate the cases $\alpha=1$ and $\alpha=0$ by means of the injection $\phi$ (Section~\ref{1var}) again. In fact, it preserves (a2) and (b2) as it did for the first variation. Also, the cardinalities of the successors sets $T$ and $T'$ associated with $h$ (subject to (a2) and (b2)) and $\phi(h)$, respectively, are equal by the following equations, where $V=\{r\in R\ |\ bdr=dr\}$.
\begin{align*}
|T| &= \left\vert\bigcup\limits_{r\in V}cr+(((c+d)K+v+dr)\cap K)\e\right\vert\\
&= \left\vert\bigcup\limits_{r\in V}(c+d)r+(((c+d)K+v+dr)\cap K)\x\right\vert
= |T'|
\end{align*}
For the Renaissance dichotomy, in the following section we observe that the dual numbers model has all contrapuntal symmetries in the image of $\phi$, so, as in Section~\ref{1var}, the contrapuntal symmetries for $\alpha=1$ are the inverse images of those for $\alpha=0$. But, by collecting these inverse images, we obtain the same original set of symmetries up to indeterminate names $\x,\e$ (see Table~\ref{tab:adsufin}), so \textit{the contrapuntal symmetries coincide for $\alpha=0$ and $\alpha=1$ in the case of the Renaissance dichotomy}.

\subsection{Third variation}

If we require the local characterization property on the deformed dichotomy for all fibers, which amounts to a \textit{global property}, then we change condition 2 of Definition~\ref{def} for the following one.
\begin{itemize}
\item For all $z\in R$, $P(z)$ holds.
\end{itemize}

The respective simplification is the following theorem.
\begin{theorem}\label{redu3} The admitted successors of $k\x\in K[\x]$ are the elements of the sets $S$ of the form
\[h(K[\x])\cap K[\x],\]
where $h\in H$ and 
\begin{itemize}
\item[(a3)] $k\x\in h(D[\x])$, 
\item[(b3)] for all $z\in R$, $P(z)$ holds, and 
\item[(c3)] the cardinality of $S$ is maximum among all $h\in H$ satisfying (a3) and (b3). 
\end{itemize}
\end{theorem}

According to the proof of Theorem~\ref{equivcommu}, if $h=e^{v\x}(c+d\x)$, the condition (b3) just says that (b2) is true and that $bdc^{-1}z=dc^{-1}z$ holds for all $z\in R$. Hence, (b3) is equivalent to (b2) and $bd=d$.

For the dual numbers case, we next prove that \textit{the second and third variations coincide if $R=\mathbb{Z}_n$ and each consonance has at least an admitted successor in the third variation}, like in the case of the Renaissance dichotomy; see Table~\ref{tab:adsufin}. On the one hand, if $h$ satisfies (a3) and (b3), then it satisfies (a2) and (b2) and the successors sets (cardinalities) are the same for both definitions. In particular, each consonance has at least an admitted successor in the second variation too. On the other hand, if $h\in H$ with $h=e^{v\e}(c+d\e)$ is a contrapuntal symmetry in the second variation, then the set of all $z\in R$ such that $P(z)$ holds, that is, such that
\[bdz=dz,\]
is a principal ideal\footnote{In fact, the solution set of $ax=0$ in $\mathbb{Z}_n$ is $\langle o\rangle$, where $o=n/\gcd(a,n)$.} $\langle o\rangle$ in $R$. We have two possibilities. If $\langle o\rangle=R$, then $h$ satisfies (a3) and (b3). Otherwise, $\langle o\rangle\subsetneq R$ and the set $T$ (which is nonempty) from Theorem~\ref{redu2} satisfies
\begin{align*}
|T| &= \left\vert\bigcup_{z\in \langle o\rangle} z+((cK+v+dc^{-1}z)\cap K)\e\right\vert\\
&= \sum_{z\in \langle o\rangle}\left\vert (cK+v+dc^{-1}z)\cap K\right\vert\\
  &< \rho\sum_{z\in \langle o\rangle}\left\vert (cK+v+dc^{-1}z)\cap K\right\vert\\  
  &= \sum_{r\in R}\left\vert (cK+v+dc^{-1}or)\cap K\right\vert\\  
    &=\left\vert\bigcup_{r\in R} r+((cK+v+dc^{-1}or)\cap K)\e\right\vert\\
    &= |e^{v\e}(c+do\e)(K[\e])\cap K[\e]|,
\end{align*}
where $\rho=|\Ker(o\cdot-)|=|R|/|\langle o\rangle|>1$ and $bdo=do$. We deduce that $h$ satisfies (a3) and (b3), or there is $h'$ satisfying (a3) and (b3) such that its successor set has a strictly greater cardinality than the successors set of $h$. Our observations imply that the contrapuntal symmetries coincide for the second and third variations. 

In the Renaissance dichotomy case, by the computation of those symmetries (Table~\ref{tab:adsufin}), they are in the image of $\phi$. The latter preserves the conditions (a3) and (b3) and successors sets cardinalities from $\alpha=1$ to $\alpha=0$, so the contrapuntal symmetries coincide for $\alpha=1$ and $\alpha=0$, as in Section~\ref{2var}.

\begin{table}
\begin{center}
\begin{tabular}{|c|c|c|}
\hline
$k$ & $|h(K[\e])\cap K[\e]|$ & $h$ \\ \hline
\multirow{3}{*}{$0$} & \multirow{3}{*}{$48$} &  $e^{6\e}(1+6\e)$ \\ 
 & & $e^{6\e}(7+6\e)$  \\ 
  &  &   $e^{11\e}11$ \\ \hline
$3$& $48$ &  $e^{8\e}5$ \\ \hline
\multirow{2}{*}{$4$}& \multirow{2}{*}{$48$} & $e^{6\e}(1+6\e)$\\ 
 & & $e^{6\e}(7+6\e)$\\ \hline
$7$ & $60$ & $e^07$ \\ \hline
\multirow{5}{*}{$8$} & \multirow{3}{*}{$48$} & $e^{3\e}7$\\ 
                     & & $e^{6\e}(1+6\e)$\\ 
                    &  & $e^{6\e}(7+6\e)$ \\\hline
\multirow{2}{*}{$9$}& \multirow{2}{*}{$48$} & $e^{11\e}11$\\ 
       &  & $e^{3\e}7$  \\ 
       &  & $e^{8\e}5$  \\ \hline
\end{tabular}
\end{center} 
\caption{Contrapuntal symmetries for the Renaissance dichotomy in the third variation and the dual numbers case. The computation, by hand, is akin to that of Table~\ref{tab:adsu}---now we just add the condition $5d=d$ for contrapuntal symmetries, that is, $d=0,3,6,9$ or $\rho=3,6,12$. }
\label{tab:adsufin}
\end{table}

To summarize, \textit{the second and third variations coincide for the Renaissance dichotomy}. 

\section{Appendix}\label{ap}
\addcontentsline{toc}{section}{\nameref{ap}}

\begin{theorem}[\textbf{Rearrangement inequality, variation}]\label{r.i.}
Suppose given an ordered list of real numbers
\[a_1\geq a_2\geq \dots \geq a_n.\]
The inequality
\[\sum\limits_{i=1}^n a_i^2\geq \sum\limits_{i=1}^n a_ia_{\sigma(i)}\]
holds for each permutation $\sigma$ of $\{1,\dots, n\}$. Moreover, the equality holds if and only if $(a_1,\dots,a_n)=(a_{\sigma(1)},\dots, a_{\sigma(n)})$.
\end{theorem}
\begin{proof}
Induction on $n$. The case $n=1$ is just the equality $a_1^2=a_1^2$. Now suppose the result for $n-1$. We consider two cases. 

\textit{First case}: $a_{\sigma(1)}=a_1$. We can assume $\sigma(1)=1$ (switch $\sigma(1)$ and $1$) since this do not alter the sum $\sum\limits_{i=1}^n a_ia_{\sigma(i)}$. In this case, 
\[\sum\limits_{i=1}^n a_ia_{\sigma(i)}=a_1^2+\sum\limits_{i=2}^n a_ia_{\sigma(i)}=a_1^2+\sum\limits_{i=1}^{n-1} a'_ia'_{\sigma'(i)},\]
where $a'_i=a_{i+1}$ and $\sigma'(i)=\sigma(i+1)-1$ for each $i=1, \dots , n-1$. Thus, by applying the induction hypothesis,
\[\sum\limits_{i=2}^{n} a_ia_{\sigma(i)}\leq \sum\limits_{i=2}^{n} a_i^2\]
and the equality holds if and only if $(a_2,\dots ,a_n)=(a_{\sigma(2)},\dots, a_{\sigma(n)})$. Hence, by adding $a_1^2$ on both sides, 
\[\sum\limits_{i=1}^{n} a_ia_{\sigma(i)}\leq \sum\limits_{i=1}^{n} a_i^2\]
and the equality holds if and only if $(a_1,\dots ,a_n)=(a_{\sigma(1)},\dots, a_{\sigma(n)})$.

\textit{Second case}: $a_{\sigma(1)}\neq a_1$. Note that there is $k$ with $a_{\sigma(k)}=a_1$ such that $a_k< a_1$; otherwise $\{k\ |\ a_{\sigma(k)}=a_1\}\subseteq \{k\ |\ a_k=a_1\}$ and, since these sets have the same finite cardinality, they are equal, which is a contradiction since $1$ is in the latter but not in the former. Thus, 
\[\sum\limits_{i=1}^n a_ia_{\sigma(i)}=a_1a_{\sigma(1)}+a_k a_{\sigma(k)}+\sum\limits_{i\in\{1,\dots, n\}\setminus \{1,k\}} a_ia_{\sigma(i)}\]
and, if $\sigma'$ is obtained from $\sigma$ by switching $\sigma(1)$ and $\sigma(k)$, then
\[\sum\limits_{i=1}^n a_ia_{\sigma'(i)}-\sum\limits_{i=1}^n a_ia_{\sigma(i)}=(a_1-a_k)(a_{\sigma(k)}-a_{\sigma(1)})=(a_1-a_k)(a_1-a_{\sigma(1)})>0.\]
This means that 
\[\sum\limits_{i=1}^n a_ia_{\sigma(i)}<\sum\limits_{i=1}^n a_ia_{\sigma'(i)}\leq \sum\limits_{i=1}^n a_i^2,\]
where the right-hand inequality corresponds to the first case. Finally, note that here the inequality is always strict and $(a_1,\dots ,a_n)\neq(a_{\sigma(1)},\dots, a_{\sigma(n)})$. 
\end{proof}

\section*{Funding}
\addcontentsline{toc}{section}{Funding}

This work was supported by Programa de Becas Posdoctorales en la UNAM 2019, which is coordinated by Direcci\'{o}n de 
Asuntos del Personal Acad\'{e}mico (DGAPA) at Universidad Nacional Aut\'{o}noma de M\'{e}xico. 

\section*{Disclosure statement}
\addcontentsline{toc}{section}{Disclosure statement}
No potential conflict of interest was reported by the authors.

\bibliography{MySubmissionBibTexDatabase}

\begin{thebibliography}{10}

\bibitem{OAAAthesis}
Octavio~A. Agust\'{i}n-Aquino.
\newblock {\em Extensiones Microtonales de Contrapunto (English Version)}.
\newblock PhD thesis, Universidad Nacional Aut\'{o}noma de M\'{e}xico, 2011.

\bibitem{octaviotod}
Octavio~A. Agust\'{\i}n-Aquino.
\newblock Tod der {H}omophonie! {E}s lebe der {K}ontrapunktsatz!
\newblock \url{https://www.youtube.com/watch?v=2nh8xoLwC0o}, November 2015.

\bibitem{Octavio}
Octavio~A. Agust{\'\i}n-Aquino, Julien Junod, and Guerino Mazzola.
\newblock {\em Computational Counterpoint Worlds: Mathematical Theory,
  Software, and Experiments}.
\newblock Computational Music Science. Springer International Publishing, Cham,
  Switzerland, 2015.

\bibitem{Scriabinworld}
Octavio~A. Agust{\'i}n-Aquino and Guerino Mazzola.
\newblock {Contrapuntal Aspects of the Mystic Chord and Scriabin’s Piano
  Sonata No. 5}.
\newblock In Mariana Montiel, Francisco Gomez-Martin, and Octavio~A.
  Agust{\'i}n-Aquino, editors, {\em Mathematics and Computation in Music},
  pages 3--20. Springer International Publishing, Cham, Switzerland, 2019.

\bibitem{Musicology}
Juan~S. Arias-Valero, Octavio~A. Agust\'{\i}n-Aquino, and Emilio Lluis-Puebla.
\newblock {Musicological, Computational, and Conceptual Aspects of
  First-Species Counterpoint Theory}.
\newblock {\em Preprint}.

\bibitem{Fraleigh}
John~B. Fraleigh.
\newblock {\em A First Course in Abstract Algebra}.
\newblock Pearson Education, New York, 7th edition, 2003.

\bibitem{Fux}
Johann~J. Fux, Alfred Mann, and John Edmunds.
\newblock {\em The Study of Counterpoint from Johann Joseph Fux’s Gradus ad
  Parnassum}.
\newblock W. W. Norton, New York, 1965.

\bibitem{Junod}
Julien Junod.
\newblock {\em Counterpoint Worlds and Morphisms : A Graph-Theoretical Approach
  and Its Implementation on the Rubato Composer Software}.
\newblock PhD thesis, University of Zurich, 2010.

\bibitem{Lam}
Tsit~Y. Lam.
\newblock {\em A First Course in Noncommutative Rings}.
\newblock Springer-Verlag, New York, {Second} edition, 2001.

\bibitem{Beau}
Guerino Mazzola.
\newblock {\em La V\' erit\' e du Beau dans la Musique}.
\newblock Editions Delatour France, Paris, 2007.

\bibitem{MazzolaTopos}
Guerino Mazzola et~al.
\newblock {\em The Topos of Music: Geometric Logic of Concepts, Theory, and
  Performance}.
\newblock Birkh\"auser Verlag, Basel, 2002.

\bibitem{Inicount}
Guerino Mazzola, Heinz-Gregor Wieser, Vreni Brunner, and Daniel Muzzulini.
\newblock {A Symmetry-Oriented Mathematical Model of Classical Counterpoint and
  Related Neurophysiological Investigations by Depth EEG}.
\newblock {\em Computers \& Mathematics with Applications}, 17(4--6):539--594,
  1989.

\end{thebibliography}

\end{document}